\documentclass[11pt,a4paper,reqno]{amsart}

\usepackage{a4wide}

\usepackage{amssymb,amsmath,amsthm,mathrsfs}
\usepackage{graphicx}
\usepackage{float}
\usepackage{colortbl}
\usepackage{enumerate}
\usepackage{upref}
\usepackage{stackrel}
\usepackage{enumitem}
\usepackage{bbm}
\usepackage[bbgreekl]{mathbbol}
\usepackage{dsfont}

\makeatletter
\def\amsbb{\use@mathgroup \M@U \symAMSb}
\makeatother

\usepackage[bookmarks=false]{hyperref}
\usepackage[T2A,T1]{fontenc}
\DeclareSymbolFont{cyrillic}{T2A}{cmr}{m}{n}
\DeclareMathSymbol{\D}{\mathalpha}{cyrillic}{196}

\usepackage{accents}
    
\usepackage[colorinlistoftodos]{todonotes}

\usepackage{amsfonts} 
\DeclareMathSymbol{\shortminus}{\mathbin}{AMSa}{"39}

\usepackage{tikz}
\usetikzlibrary{arrows}
\usetikzlibrary{decorations.pathreplacing}

\usepackage{chngcntr}
\counterwithin{figure}{section}

\theoremstyle{plain}
\newtheorem{theorem}{Theorem}[section]
\newtheorem*{theorem*}{Theorem}
\newtheorem{lemma}[theorem]{Lemma}

\newtheorem{proposition}[theorem]{Proposition}

\theoremstyle{definition}
\newtheorem{definition}[theorem]{Definition}

\newtheorem*{condition}{Condition}

\theoremstyle{remark}
\newtheorem{remark}[theorem]{Remark}

\usepackage{enumitem}
\makeatletter
\def\namedlabel#1#2{\begingroup
   #2%
 \def\@currentlabel{#2}%
   \phantomsection\label{#1}\endgroup
}

\def\R{\ensuremath{\amsbb R}}
\def\N{\ensuremath{\amsbb N}}

\def\e{{\ensuremath{\rm e}}}

\def\p{\ensuremath{\amsbb P}}
\def\O{\ensuremath{\text{O}}}
\def\oo{\ensuremath{\text{o}}}

\def\A{\ensuremath{A_n^{(q)}}}

\def\cZ{\mathcal{Z}}

\def\1{{\mathbbm 1}}

\def\XX{\mathbb{X}}
\def\uu{\mathbb u}
\def\MM{\mathbb M}
\def\ttau{\bbtau}
\def\ttheta{\bbtheta}
\def\tt{\mathbb t}
\def\PPsi{\mathbb \Psi}
\def\FF{\mathbb F}

\def\ie{{\em i.e.}, }
\def\iid{{i.i.d.\ }}
\def\df{d.f.\ }

\def\dist{\ensuremath{\text{dist}}}

\def\es{{\emptyset}}
\def\sm{\setminus}
\def\eps{\varepsilon}

\numberwithin{equation}{section}

\begin{document}

\title[Multivariate extreme values for dynamical systems]{Multivariate extreme values for dynamical systems}

\author[R. Aimino]{Romain Aimino}
\address{Romain Aimino\\ Centro de Matem\'{a}tica da Universidade do Porto\\ Rua do
Campo Alegre 687\\ 4169-007 Porto\\ Portugal}
\email{\href{mailto:romain.aimino@fc.up.pt}{romain.aimino@fc.up.pt}}
\urladdr{\url{http://www.fc.up.pt/pessoas/romain.aimino/}}

\author[A. C. M. Freitas]{Ana Cristina Moreira Freitas}
\address{Ana Cristina Moreira Freitas\\ Centro de Matem\'{a}tica \&
Faculdade de Economia da Universidade do Porto\\ Rua Dr. Roberto Frias \\
4200-464 Porto\\ Portugal} \email{\href{mailto:amoreira@fep.up.pt}{amoreira@fep.up.pt}}
\urladdr{\url{http://www.fep.up.pt/docentes/amoreira/}}

\author[J. M. Freitas]{Jorge Milhazes Freitas}
\address{Jorge Milhazes Freitas\\ Centro de Matem\'{a}tica \& Faculdade de Ci\^encias da Universidade do Porto\\ Rua do
Campo Alegre 687\\ 4169-007 Porto\\ Portugal}
\email[Corresponding author]{\href{mailto:jmfreita@fc.up.pt}{jmfreita@fc.up.pt}}
\urladdr{\url{http://www.fc.up.pt/pessoas/jmfreita/}}

\author[M. Todd]{Mike Todd}
\address{Mike Todd\\ Mathematical Institute\\
University of St Andrews\\
North Haugh\\
St Andrews\\
KY16 9SS\\
Scotland \\} \email{\href{mailto:m.todd@st-andrews.ac.uk }{m.todd@st-andrews.ac.uk }}
\urladdr{\url{https://mtoddm.github.io/}}

\thanks{All authors were partially financed by Portuguese public funds through FCT  -- Funda\c{c}\~ao para a Ci\^encia e a Tecnologia, I.P., in the framework of the projects PTDC/MAT-PUR/4048/2021, 2022.07167.PTDC and CMUP's project with reference UIDB/00144/2020. JMF and MT thank the Erwin Schr\"odinger Institute (meeting ZPS24) where part of this work was done.
}

\date{\today}

\keywords{Multivariate extremes, copulas, extremal index} 
\subjclass[2020]{37A50, 37A25, 37B20, 60G70, 62H05}


\begin{abstract}
We establish a theory for multivariate extreme value analysis of dynamical systems. Namely, we provide conditions adapted to the dynamical setting which enable the study of dependence between extreme values of the components of $\R^d$-valued observables evaluated along the orbits of the systems. We study this cross-sectional dependence, which results from the combination of a spatial and a temporal dependence structures. We give several illustrative applications, where concrete systems and dependence sources are introduced and analysed. 
\end{abstract}

\maketitle

\section{Introduction}

The study of rare events for dynamical systems is recent but has experienced a vast development in the last decade, partly motivated by applications to climate dynamics, where dynamical systems (such as the Lorenz models) provide accurate descriptions of meteorological phenomena. This development has been anchored in a connection between the observation of rare events, detected by the appearance of extreme values, and the recurrence properties of sensitive regions, under the action of the underlying dynamics. The main idea is that chaotic systems lose memory quickly which makes their orbits behave like random asymptotically independent observations. This strategy has been successfully applied to prove the existence of limit theorems regarding the distribution of the extremal order statistics, point processes, records, as well as ergodic averages of heavy tailed observables.

In some sense, the study of extreme events for dynamical systems has only recently caught up with the state of the art of univariate Extreme Value Theory. However, since the 1980s, many extreme value theorists have moved from the univariate theory to the multivariate context, where one is concerned with extremes in a multivariate random sample, \ie events for which at least one of the components reaches exceptionally high (or low) values. Of course, focusing on the behaviour of one of the components is the subject of univariate extreme value theory. The main point of multivariate extreme value analysis is understanding the interplay between the extremes in the different components. This insight is of crucial importance in climate dynamics, where the influence between extremal observations of different variables (such as pressure and temperature) as well as their  spatial and temporal dependence is vital for predicting extreme weather events. The dependence structure of extreme phenomena is pivotal for anticipating compound risks such as in the case of the 2022 European drought addressed in \cite{FPB23} or the study of co-occurring extremes in \cite{MF23} or multivariate rainfall time series in \cite{BN23}, just to give some recent examples.

The main goal of this paper is to introduce the first (to our knowledge) theoretical results on multivariate extreme value analysis for dynamical systems, thus closing the gap in theories described above. The differences in proof techniques between the univariate and multivariate cases are few, so the technical work presented in this paper is brief: our focus is on beginning to put the multivariate approach in the dynamical setting on a firm basis and presenting some elementary examples of its application.  Unlike independent sequences or the strong mixing extremal properties of the stationary processes from the classical literature, the processes arising from dynamical systems require a more flexible time dependence structure, which we provide here. We will then study the cross-sectional dependence of the vector valued observables evaluated along the orbits of the systems, which will describe the spatial relationships where an extreme observation of one of the components is responsible for the appearance of other extreme observation in another component, but we will also analyse how the short recurrence properties may introduce a source of time dependence, so that the co-occurrence of extremes of different quantities appear slightly out of sync in time. This phenomenon is directly connected with clustering of extremal observations which can be viewed through an extremal index function: here we give a dynamically natural formula for this. Interestingly, while the spatial dependence pertaining to the cross-sectional relations between the components was copiously studied, the literature regarding to the time dependence and the extremal index function is relatively scarce. One of the advantages of this dynamical approach to multivariate extremes is that we can see, in a natural and simple way, how the local (or fast) recurrence properties of the maximal regions contribute to the overall dependence structure of extremes, making it one of the interests of the present paper.

We emphasise that the multivariate extremes perspective here is completely new in the dynamical systems setting, although there have been works considering higher dimensional variables such as \cite{FV18} or multidimensional point processes accounting for the extremes of $\R^d$-valued dynamically defined stationary processes as in \cite{FFM20, FFT25}. While in the former case, the authors take observables corresponding to the distance to the diagonal, reducing it to a univariate problem, in the latter case the relations between the $d$-components are not considered.

The paper is organised as follows. In Section~\ref{sec:MEVA} we make a brief introduction to Multivariate Extreme Value Analysis, present the pertinent concepts and objects and then, in Section~\ref{subsec:MEVA-dynamics}, we give the main theoretical results which enable its application to dynamical systems. In Section~\ref{sec:examples}, we give some illustrative applications, introducing some mechanisms to create both spatial and temporal dependence, based on particular choices of observables and dynamics so that the recurrence properties of the respective maximal sets give rise to different dependence profiles. We remark that the examples presented and the respective mechanisms are meant to illustrate the potential of the theory and of its applications, so as to make a clear contribution in both the dynamical and the extreme settings: there is ample scope to elaborate further on these, but here we are focused on presenting the main ideas in a simple way.   Our main focus is on stable dependence functions (and Pickands dependence functions), so we explain the connection of this theory to copulas in a short appendix.

\textit{Notation:} Given sequences $(a_n)_n$ and $(b_n)_n$ where $b_n\neq 0$ for all $n$, we will sometimes write $a_n=\O (b_n)$ if there is some $C>0$ such that $a_n\le Cb_n$; write $a_n= \oo(b_n)$ if $\lim_{n\to \infty}\frac{a_n}{b_n}=0$; and write  $a_n\sim b_n$ if $\lim_{n\to \infty}\frac{a_n}{b_n}=1$.

\section{Multivariate extreme value analysis}
\label{sec:MEVA}

Let $(\mathcal X,\mathcal B_{\mathcal X}, \mu, f)$ be a discrete time dynamical system, where $\mathcal X$ is a compact metric space, 
 $\mathcal B_{\mathcal X}$ is its Borel $\sigma$-algebra, $f:{\mathcal X}\to{\mathcal X}$ is a measurable map and $\mu$ is an $f$-invariant probability measure, \ie $\mu(f^{-1}(B))=\mu(B)$ for all $B\in\mathcal B_{\mathcal X}$. Let $\mathbb \Psi:{\mathcal X}\to\bar{\R}^d$ be an observable (measurable) function and define the stochastic process $\XX_0, \XX_1,\ldots$ given by
\begin{equation}
\label{eq:dynamics-SP}
\XX_n=\mathbb \Psi\circ f^n,\qquad\text{for every $n\in\N_0$}.
\end{equation}
We will use the notation $\XX_n=(X_{n1},\ldots, X_{nd})$ and $\mathbb  \Psi=(\psi_1,\ldots,\psi_d)$ whenever we need to refer to the respective components specifically.

In order to establish notation, we will use blackboard bold for vectors or vector valued functions taking values in $\R^d$. In particular, $\mathbb 0, \mathbb 1$ correspond to the vectors in $\R^d$ with all entries equal to $0$ and $1$, respectively. Operations and functions applied to vectors are to be interpreted componentwise, so for example: for $\tt,\ttau\in(0,\infty)^d$ and $c>0$, we write  
$$
\e^{-\ttau}:=\left(\e^{-\tau_1},\ldots,\e^{-\tau_d}\right),\quad (\ttau+\tt)^c:=((\tau_1+t_1)^c,\ldots, (\tau_d+t_d)^c),\quad \ttau/\tt:=\left(\frac{\tau_1}{t_1},\ldots,\frac{\tau_d}{t_d}\right).
$$

Our main goal is to study the multivariate extremal behaviour of such stochastic processes arising from chaotic dynamics. For that purpose we introduce the componentwise maxima sequence
$$
\MM_n:=(M_{n1},\ldots,M_{nd}),\quad\text{where}\quad M_{nj}:=\max_{i=0,\ldots,n-1}X_{ij},\quad\text{for}\quad j=1,\ldots,d.
$$ 

For an asymptotic frequency vector $\ttau=(\tau_1,\ldots,\tau_d)\in[0,+\infty)^d\setminus \mathbb 0$, we consider a sequence of normalising vectors $(\uu_n(\ttau))_{n}$ such that $ \uu_n(\ttau)=(u_{n1}(\tau_1),\ldots,u_{nd}(\tau_d))$ and
\begin{equation}
\label{eq:def_un}
\lim_{n\to\infty}n\mu(X_{0j}>u_{nj}(\tau_j))=\lim_{n\to\infty}n\mu\left(\left\{x\in\mathcal X\colon \psi_j(x)>u_{nj}(\tau_j)\right\}\right)=\tau_j, \quad  j=1,\ldots,d.
\end{equation}

Let $\tt\in(0,1)^d$ be such that $\tt=\e^{-\ttau}$ (or equivalently $\ttau=-\log \tt$). 
We aim to find a multivariate extreme value distribution function (d.f.) $H$ supported on $[0,1]^d$ and such that
\begin{equation}
\label{eq:def_H}
\lim_{n\to\infty}\mu(\MM_n\leq \uu_n(-\log \tt))=H(\tt) .
\end{equation}

We introduce some notation we will use later: $\XX_n\not\leq\uu_n(\ttau)$ means that there exists $1\le j\le d$ such that $X_{nj}>u_{nj}(\tau_j)$.

\subsection{The classical setting}
\label{subsec:MEVA-classic}

Let $\hat \XX_0, \hat \XX_1,\ldots$ denote an associated \iid sequence of random vectors with $\hat\XX_0=\mathbb \Psi$ (i.e., $\mu(\hat\XX_i>\tt) =\mu(\{x:\PPsi(x)>\tt\})$). Also define the respective sequence of partial maxima vectors $(\hat \MM_n)_n$ with components $\hat M_{nj}$, $j=1,\ldots,d$ and multivariate extreme value \df $\hat H$ analogously as for $\XX_n$. Note that by \eqref{eq:def_un} and the definition of $\tt$ we have that $\hat H$ has uniform marginals, namely,
$$
\hat H_j(t_j)=\lim_{n\to\infty}\mu\left(\hat M_{nj}\leq u_{nj}(-\log t_j)\right)=\lim_{n\to\infty}\left(\mu\left(\psi_j\leq u_{nj}(-\log t_j)\right)\right)^n=t_j, \quad j=1,\ldots,d.
$$
This fact and the fact that the joint distribution function of $\hat \MM_n$ is the $n$-th power of that of $\hat \XX_0$ allow us to obtain that $\hat H$ satisfies the homogeneity property (further elaborated in \eqref{eq:copula-homogeneity}):
\begin{equation}
\label{eq:homogeneity-H-hat}
\hat H(\tt^c)=\left(\hat H(\tt)\right)^c \qquad \text{for all}\quad \tt\in[0,1]^d\quad\text{and}\quad c>0.
\end{equation}

The main interest of multivariate analysis is understanding the dependence between the various components of the multivariate observations. For this purpose several devices such as dependence functions and intensity measures have been introduced and thoroughly studied (see \cite{FHR11,S12}, for example). Our main tool here will be the stable tail dependence function  (or the $D$-norm in the terminology of \cite{FHR11}, where $D$ is the Pickands dependence function). In order to introduce it in a simple and compact way, we assume that the marginals $H_j$ are continuous and $H_j^{-1}$ denotes the respective generalised inverse distribution function ($H_j^{-1}(u)=\inf\{t\colon\; H_j(t)=u\}$). This will be the case in all our applications (recall that in the \iid case $\hat H_j$ is the uniform distribution). 
\begin{definition}
Assuming that \eqref{eq:def_H} holds, we define the respective \emph{stable dependence function} for each $\ttau\in [0,\infty)^d$ as
\begin{equation*}
\label{eq:gamma-hat-def}
\Gamma(\ttau):=-\log H \left(H_1^{-1}\left(\e^{-\tau_1}\right),\ldots,H_d^{-1}\left(\e^{-\tau_d}\right)\right).
\end{equation*}
\end{definition}

 The stable dependence function can alternatively be written as a log transformation of the limiting extreme value copula, see the appendix. 
Sometimes $\Gamma(\ttau)$ is referred to as the \emph{$D$-norm} of $\ttau$ (\cite{FHR11}). Note that in the \iid setting, since we have uniform marginals, $\hat \Gamma (\ttau)=-\log \hat H \left(\e^{-\ttau}\right)$. Let $\hat G(\ttau):=\e^{-\hat \Gamma(\ttau)}$. As in the univariate context (see \cite[Theorem~1.5.1]{LLR83}), one can show that the existence of the limit $\lim_{n\to\infty}\mu(\hat{\mathbb M}_n\leq \uu_n(\ttau))=\hat G(\ttau)$ is equivalent to the existence of the limit
\begin{equation}
\label{eq:unvector}
\lim_{n\to\infty}n\mu(\XX_0\not\leq\uu_n(\ttau))=\lim_{n\to\infty}n\mu\left(\bigcup_{j=1}^d\{X_{0j}>u_n(\tau_j)\}\right)=\hat\Gamma(\ttau).
\end{equation}
The homogeneity property \eqref{eq:homogeneity-H-hat} translates to 
\begin{equation}
\label{eq:hat-homogeneity}
\hat\Gamma(c\ttau)=c\,\hat\Gamma(\ttau)\qquad\text{and}\qquad \hat G(c\ttau)=(\hat G(\ttau))^c, \qquad \text{for all}\quad c>0.
\end{equation}
Moreover, we have that
\begin{equation}
\label{eq:bounds}
\max\{\tau_1,\ldots,\tau_d\}\leq \hat \Gamma(\ttau)\leq \tau_1+\cdots+\tau_d,
\end{equation}
where the upper bound corresponds to asymptotic component independence, while the lower bound corresponds to asymptotic perfect association.
This terminology becomes transparent if we reinterpret \eqref{eq:bounds} in terms of the copula $C_{\hat H}$ of the distribution function $\hat H$ (see Appendix \ref{sec:appendix_copulas}):
\begin{equation}
    \label{eq:bounds_copula}
     \min\{t_1, \ldots, t_d\} \ge C_{\hat H}(\tt) \ge t_1 \cdot \cdots \cdot t_d, \qquad \text{for all}\quad \tt \in [0,1]^d,
\end{equation}
 Note that both bounds correspond to realisable stable dependence functions.

When we consider stationary stochastic processes and drop the independence assumption, a new source of dependence may appear. This temporal dependence is associated to clustering of extremal multivariate observations (corresponding to at least one component taking an extremal value), which is conveniently described by a multivariate extremal index function $\theta(\ttau)$ introduced in \cite{N94}. Recall that, in the univariate case, the extremal index is a parameter $\theta_j\in[0,1]$ that appears when, for every $\tau_j\geq0$ and $(u_{nj}(\tau_j))_n$ as in \eqref{eq:def_un}, we have  
\begin{equation}
\lim_{n\to\infty}\mu(M_{nj}\leq u_{nj}(\tau_j))=\e^{-\theta_j\tau_j}.
\label{eq:Mmargs}
\end{equation}
 Note that, when this limit exists (which we express by saying that $\theta_j$ exists or is well defined) then the marginals of $H$ are such that $H_j(t_j)=t_j^{\theta_j}$. Therefore, assuming that for each $j=1,\ldots, d$, every $\theta_j$ is well defined and putting 
$$
\ttheta:=(\theta_1,\ldots,\theta_d),
$$
we have that 
$$H\left(\e^{-\ttau}\right)=\e^{-\Gamma(\ttheta\ttau)}=:G(\ttau).$$  
\begin{definition}
For every $\ttau\in [0,\infty)^d\setminus \mathbb 0$, we define the multivariate extremal index function as: 
\begin{equation}
\label{def:EI}
\theta(\ttau):=\frac{\log G(\ttau)}{\log\hat G(\ttau)}=\frac{\Gamma(\ttheta\ttau)}{\hat \Gamma (\ttau)}.
\end{equation}
\end{definition}
In \cite[Proposition~3.1]{N94}, under a distributional mixing assumption inspired by Leadbetter's $D$ condition, denoted there by $\Delta$, it was proved that $\theta(c\ttau)=\theta(\ttau)$ for all $c>0$, which means that both $\Gamma$ and $G$ share the same homogeneity property of their hat versions stated in  \eqref{eq:hat-homogeneity}, and the univariate marginal extremal indices $\theta_j$, $j=1,\ldots,d$, can be recovered from the multivariate extremal index function by setting all coordinates of $\ttau$ equal to $0$ except the $j$-th. Moreover, one can check that the bounds of \eqref{eq:bounds} also apply to $\Gamma$.

As we explain below, motivated by the dynamical applications, we will assume a mixing condition weaker than Leadbetter's, we will use a more tractable formula for the multivariate extremal index function and we will prove that these properties still hold in that context.

The homogeneity property of $\Gamma$ suggests a reduction to the $(d-1)$-dimensional unit simplex $\mathcal S_{d}=\{\bbalpha\in[0,1]^d\colon\alpha_1+\cdots+\alpha_d=1\}$: the restriction of $\Gamma$ to $\mathcal S_{d}$ is often called the \emph{Pickands dependence function} $D$:
$$
\Gamma(\ttau)=(\tau_1+\cdots+\tau_d) D\left(\alpha_1,\ldots,\alpha_{d-1}\right),\quad\text{where}\quad\alpha_j=\frac{\tau_j}{\tau_1+\cdots+\tau_d}.
$$

\subsection{A dynamically adapted approach to the study of multivariate extremes}
\label{subsec:MEVA-dynamics}

Our main goal here is to study the convergence \eqref{eq:def_H} for dynamically defined multivariate stochastic processes as in \eqref{eq:dynamics-SP} and to give a more computable formula for $\theta(\ttau)$. Convergence for general stationary stochastic processes was initiated in \cite{H89a,H89} under a distributional mixing condition very much akin to the original $D$ condition introduced by Leadbetter in the univariate context. This condition has been used ubiquitously to study extremes of stationary random vectors. 
The problem is that this condition is not amenable to application in the dynamical setting and therefore we propose a weaker version adapted to this context motivated by our earlier works in the univariate framework, where this weakness is compensated by a condition similar to the $D^{(k)}$ condition from \cite{CHM91}, which allows clustering but forbids the concentration of clusters, so that we can still recover the existence of an extremal limit. This is accomplished using an idea introduced in \cite{FFT12} and further elaborated in \cite{FFT15}, which essentially says that we may replace the occurrence of an abnormal observation by that consisting of an abnormal observation followed by a block of regular observations, which in the dynamical setting corresponds to replacing balls by annuli. For that purpose, for $q\in\N$ and $\ttau\in[0,\infty)^d$ we define
\begin{equation*}
\label{eq:A-def}
\A(\ttau)=\left\{\XX_0\not\leq \uu_n(\ttau) \right\}\cap f^{-1}(\mathbb M_q\leq\uu_n(\ttau))=\left\{\XX_0\not\leq \uu_n(\ttau), \XX_1\leq \uu_n(\ttau), \ldots, \XX_q\leq \uu_n(\ttau)\right\}.
\end{equation*}
We also set $A^{(0)}_n(\ttau):=\left\{\XX_0\not\leq \uu_n(\ttau)\right\}$. 
Let $B\in\mathcal B_\mathcal X$ be an event. For some $s,\ell\in\N_0$, we define:
\begin{equation*}
\label{eq:W-def}
\mathscr W_{s,\ell}(B)=\bigcap_{i=s}^{s+\max\{\ell-1,\ 0\}} f^{-i}(B^c).
\end{equation*}
We will write $\mathscr W_{s,\ell}^c(B):=(\mathscr W_{s,\ell}(B))^c$. 
Observe that 
$
\mathscr W_{0,n}(A^{(0)}_n(\ttau))=\{\MM_n\leq \uu_n(\ttau)\}.
$

We state now the two main conditions on the time dependence structure of the process.
\begin{condition}[$\D(\uu_n)$]\label{cond:D} We say that $\D(\uu_n)$ holds for the sequence $\XX_0,\XX_1,\ldots$ if for every  $\ell,t,n\in\N$ and $q\in\N_0$,
\begin{equation*}\label{eq:condD}
\left|\mu\left(\A(\ttau)\cap
 \mathscr W_{t,\ell}\left(\A(\ttau)\right) \right)-\mu\left(\A(\ttau)\right)
  \mu\left(\mathscr W_{0,\ell}\left(\A(\ttau)\right)\right)\right|\leq \gamma(q,n,t),
\end{equation*}
where $\gamma(q,n,t)$ is decreasing in $t$ for each $q, n$ and, for every $q\in\N_0$, there exists a sequence $(t_n)_{n\in\N}$ such that $t_n=\oo(n)$ and
$\lim_{q\to\infty}\limsup_{n\to\infty}n\gamma(q,n,t_n)=0$.
\end{condition}
The condition above is a mixing type condition obtained from adjusting to the multivariate setting the homonymous condition from our previous work (see \cite{FFT15} for example), which has the advantage of not imposing a uniform bound on $q$.  This means that, as with its univariate version, it will follow easily for systems with sufficiently fast decay of correlations as in all the examples considered below. 

We now introduce the corresponding clustering separation condition.
For some fixed $q\in\N_0$, consider the sequence $(t_n)_{n\in\N}$, given by condition  $\D(\uu_n)$ and let $(k_n)_{n\in\N}$ be another sequence of integers such that 
\begin{equation}
\label{eq:kn-sequence}
k_n\to\infty\quad \mbox{and}\quad  k_n t_n = \oo(n).
\end{equation}

\begin{condition}[$\D'(\uu_n)$]\label{cond:D'q} We say that $\D'(\uu_n)$
holds for the sequence $\XX_0,\XX_1,\XX_2,\ldots$ if there exists a sequence $(k_n)_{n\in\N}$ satisfying \eqref{eq:kn-sequence} and such that
\begin{equation*}
\label{eq:D'}
\lim_{q\to\infty}\Delta^{(q)}(\uu_n(\ttau)):=\lim_{q\to\infty}\limsup_{n\rightarrow\infty}\,n\sum_{j=q+1}^{\lfloor n/k_n\rfloor-1}\mu\left( \A(\ttau)\cap f^{-j}\left(\A(\ttau)\right)
\right)=0.
\end{equation*}
\end{condition}
The cluster separating condition we give here uses a double limit as the original anti-clustering condition $D'$ from Leadbetter, as opposed to the more general version we used in \cite{FFT25}, with a diverging $(q_n)_n$ sequence. This option guarantees a cleaner and easier proof of some of the properties of $\theta(\ttau)$ and $\Gamma(\ttau)$, without compromising the applications since it will be satisfied in all examples considered below. Note that $\D'(\uu_n)$ does not forbid the appearance of clustering, it just imposes the clusters to appear sufficiently well separated in the time line.

We give now a formula for the extremal index function, which together with the dependence conditions above, will allow us to verify that it actually provides an alternative definition, which coincides with the original one and enjoys the same properties:
\begin{equation}
\label{def:thetan}
\theta(\ttau)=
\lim_{q\to\infty}\lim_{n\to\infty}\frac{\mu\left(\A(\ttau)\right)}{\mu\left(A^{(0)}_n(\ttau)\right)}.
\end{equation}
Note that as in condition $\D'(\uu_n)$ we use a double limit, which is also the case in \cite{P97,HV20}, for example.

\begin{theorem}
\label{thm:distributional-convergence}
Let $\XX_0,\XX_1,\ldots$ be a stationary multivariate stochastic process as in \eqref{eq:dynamics-SP} and for $\ttau\in[0,\infty)^d\setminus\mathbb 0$, let  $\uu_n(\ttau)$ be a sequence such that both \eqref{eq:def_un} and \eqref{eq:unvector} hold, for some $\hat\Gamma(\ttau)$. Assume further that conditions $\D(\uu_n)$ and $\D'(\uu_n)$ hold and that $\theta(\ttau)$ given by \eqref{def:thetan} is well defined. Then, we have
$$
\lim_{n\to\infty}\mu(\MM_n\leq \uu_n(\ttau))=\e^{-\theta(\ttau)\hat\Gamma(\ttau)}=\e^{-\Gamma(\ttheta\ttau)}=G(\ttau),
$$
where $G(\ttau)$ and $\Gamma(\ttau)$ satisfy the homogeneity property stated in \eqref{eq:hat-homogeneity}. Moreover, the marginal univariate extremal indices $\theta_j$, $j=1,\ldots,d$, can be recovered from $\theta(\ttau)$ by setting all coordinates of $\ttau$ equal to $0$ except for the $j$-th and $\Gamma$ satisfies \eqref{eq:bounds}.
\end{theorem}

\begin{proof}We follow the proof of \cite[Corollary~2.4]{FFT15} closely, whose main steps we recap here since the conditions $\D(\uu_n)$ and $\D'(\uu_n)$ are somewhat different because they involve double limits.
We start by noting that by direct application of \cite[Proposition~2.7]{FFT15} we have
\begin{equation*}
\label{eq:estimate1}
\left|\mu(\MM_n\leq \uu_n(\ttau))-\mu\left(\mathscr W_{0,n}(\A(\ttau))\right)\right|\leq q\mu(A^{(0)}_n(\ttau)\setminus\A(\ttau))\xrightarrow[n\to\infty]{}0.
\end{equation*}
Now, from \cite[Proposition~2.10]{FFT15} we obtain
\begin{align*}
\label{eq:estimate2}
\left|\mu\left(\mathscr W_{0,n}(\A(\ttau))\right)-\left(1-\left\lfloor\frac{n}{k_n}\right\rfloor\mu(\A(\ttau))\right)^{k_n}\right|&\leq \Upsilon(q,n):=2k_nt_n\mu(A^{(0)}_n(\ttau))+n\gamma(q,n,t_n)\\
&\qquad+n\sum_{j=q+1}^{\lfloor n/k_n\rfloor-1}\mu\left( \A(\ttau)\cap f^{-j}\left(\A(\ttau)\right)\right).
\end{align*}
From the definitions of the sequences $(k_n)_n,(t_n)_n$ and conditions $\D(\uu_n)$ and $\D'(\uu_n)$ we have that $\lim_{q\to\infty}\lim_{n\to\infty}\Upsilon(q,n)=0$. Moreover, since \eqref{def:EI} and \eqref{eq:unvector} hold, we have that 
$$\lim_{q\to\infty}\lim_{n\to\infty}\left(1-\left\lfloor\frac{n}{k_n}\right\rfloor\mu(\A(\ttau))\right)^{k_n}=\lim_{q\to\infty}\lim_{n\to\infty}\left(1-\frac{\mu(\A(\ttau))}{\mu(A^{(0)}_n(\ttau))}\frac{n\mu(A^{(0)}_n(\ttau))}{k_n}\right)^{k_n}=\e^{-\theta(\ttau)\hat\Gamma(\ttau)}.$$

The homogeneity of $\hat\Gamma(\ttau)$ can be derived easily from that of $C_{\hat H}$ or $\hat G$ (see for example \cite[Equations (2.4) and (2.7)]{S12}). The homogeneity of $\Gamma$ and $G$ will follow once we show that
\begin{equation*}
\theta(c\ttau)=\theta(\ttau) \qquad\text{for all} \quad \ttau\in[0,\infty)^d\setminus \mathbb 0\quad\text{and}\quad c\in(0,\infty).
\end{equation*}
For that purpose, we start by noting that for $c>0$ and $\ttau$ as before, one can replace $\uu_n(c\ttau)$ by $\uu_{\lfloor n/c\rfloor}(\ttau)$. In fact, for all $i=1,\ldots,d$,
\begin{equation*}
\label{eq:unc-un}
\lim_{n\to\infty}\frac{\mu\left(X_{0i}>u_{ni}(c\tau_i)\right)}{\mu\left(X_{0i}>u_{\lfloor n/c\rfloor i}(\tau_i)\right)}=\lim_{n\to\infty}\frac{\frac cn\lfloor\frac nc\rfloor n\mu\left(X_{0i}>u_{ni}(c\tau_i)\right)}{cn\frac1n\lfloor\frac nc\rfloor\mu\left(X_{0i}>u_{\lfloor n/c\rfloor i}(\tau_i)\right)}=\frac{ c\tau_i}{c\tau_i}=1.
\end{equation*}
This means that
$$
\lim_{n\to\infty}\frac{\left|\mu\left(X_{0i}>u_{ni}(c\tau_i)\right)-\mu\left(X_{0i}>u_{\lfloor n/c\rfloor i}(\tau_i)\right)\right|}{\mu\left(X_{0i}>u_{\lfloor n/c\rfloor i}(\tau_i)\right)}=0.
$$
Noting that $\mu\left(X_{0i}>u_{ni}(c\tau_i)\right)=\O(1/n)$ and $\mu\left(X_{0i}>u_{\lfloor n/c\rfloor i}(\tau_i)\right)=\O(1/n)$, we conclude that $\left|\mu\left(X_{0i}>u_{ni}(c\tau_i)\right)-\mu\left(X_{0i}>u_{\lfloor n/c\rfloor i}(\tau_i)\right)\right|=\oo(1/n)$.
We observe now that this implies that
\begin{align*}
\frac{\left|\mu(\XX_0\not\leq\uu_n(c\ttau))-\mu(\XX_0\not\leq\uu_{\lfloor n/c\rfloor}(\ttau))\right|}{\mu(\XX_0\not\leq\uu_n(c\ttau))}&
\leq \frac{n\sum_{i=1}^d \left|\mu\left(X_{0i}>u_{ni}(c\ttau_i)\right)-\mu\left(X_{0i}>u_{\lfloor n/c\rfloor i}(\ttau_i)\right)\right|}{n\mu(\XX_0\not\leq\uu_n(c\ttau))}\\
&\xrightarrow[n\to\infty]{} \frac{0}{\hat\Gamma(\ttau)}=0,
\end{align*}
and therefore $\mu\left(\XX_0\not\leq\uu_n(c\ttau)\right)\sim\mu\left(\XX_0\not\leq\uu_{\lfloor n/c\rfloor}(\ttau)\right)$.
A similar argument also shows that $$\mu\left(\XX_0\not\leq\uu_n(c\ttau)\cap f^{-1}\left(\mathbb M_q\leq \uu_n(c\ttau)\right)\right)\sim\mu\left(\XX_0\not\leq\uu_{\lfloor n/c\rfloor}(\ttau)\cap f^{-1}\left(\mathbb M_q\leq \uu_{\lfloor n/c\rfloor}(\ttau)\right)\right).$$
Whence,
\begin{align*}
\theta(c\ttau)&=\lim_{q\to\infty}\lim_{n\to\infty}\frac{\mu\left(\XX_0\not\leq\uu_n(c\ttau)\cap f^{-1}\left(\mathbb M_q\leq \uu_n(c\ttau)\right)\right)}{\mu(\XX_0\not\leq\uu_n(c\ttau))}\\
&=\lim_{q\to\infty}\lim_{n\to\infty}\frac{\mu\left(\XX_0\not\leq\uu_{\lfloor n/c\rfloor}(\ttau)\cap f^{-1}\left(\mathbb M_q\leq \uu_{\lfloor n/c\rfloor}(\ttau)\right)\right)}{\mu(\XX_0\not\leq\uu_{\lfloor n/c\rfloor}(\ttau))}=\theta(\ttau).
\end{align*}
For the statement regarding the marginal extremal indices, take w.l.o.g.\ $\ttau=(\tau_1,0,\ldots,0)$ for some $\tau_1>0$ and observe that $\{\MM_n\leq \uu_n(\ttau)\}=\{M_{n1}\leq u_{n1}(\tau_1)\}$, $\A(\ttau)=\{X_{01}>u_{n1}(\tau_1),X_{11}\leq u_{n1}(\tau_1),\ldots X_{q1}\leq u_{n1}(\tau_1)\}=:A_{n1}^{(q)}(\tau_1)$, $A_{n}^{(0)}(\ttau)=\{X_{01}>u_{n1}(\tau_1)\}:=A_{n1}^{(0)}(\tau_1)$. Then by applying \cite[Corollary~2.4]{FFT15}, the univariate sequence $X_{01},X_{11},\ldots$ has a univariate extremal index $\theta_1$, say, that must coincide with $\theta((\tau_1,0,\ldots,0))$.

Regarding the bounds in \eqref{eq:bounds}, the first inequality follows trivially from observing that 
$
\{\MM_n\leq \uu_n(\ttau)\}\subset\{M_{nj}\leq u_{nj}(\tau_j)\}$  for all $j=1,\ldots,d$ and that from \cite[Corollary~2.4]{FFT15} we have $\mu(M_{nj}\leq u_{nj}(\tau_j))=\e^{-\theta_j\tau_j}$.

For the second inequality, start by observing that 
$$\Gamma(\ttheta\ttau)=\theta(\ttau)\hat\Gamma(\ttau)=\lim_{q\to\infty}\lim_{n\to\infty}-\log\left(1-\left\lfloor\frac{n}{k_n}\right\rfloor\mu(\A(\ttau))\right)^{k_n}.$$
Then since $-k_n\log\left(1-\left\lfloor\frac{n}{k_n}\right\rfloor\mu(\A(\ttau))\right)\sim n\mu\left(\A(\ttau)\right)$ and $\A(\ttau)\subset\bigcup_{j=1}^d A_{nj}^{(q)}(\tau_j)$, with the necessary adjustments to \cite[Corollary~2.4]{FFT15} we get
$$
n\mu\left(\A(\ttau)\right)\leq \sum_{j=1}^d n\mu\left(A_{nj}^{(q)}(\tau_j)\right)=\sum_{j=1}^d n\mu\left(A_{nj}^{(0)}(\tau_j)\right)\frac{\mu\left(A_{nj}^{(q)}(\tau_j)\right)}{\mu\left(A_{nj}^{(0)}(\tau_j)\right)}\to\sum_{j=1}^d\tau_j \theta_j.
$$
\end{proof}

\section{Applications to dynamical systems}
\label{sec:examples}

In order to produce spatial and temporal dependence between the components regarding their asymptotic extremal behaviour, we elaborate on and adapt to the multivariate setting an idea introduced in \cite{AFFR16} and further developed in \cite{AFFR17, FFRS20}, which consists of creating a dynamic link between the points of the maximising sets of the observables. 

In \cite{FFT12}, a connection between the appearance of clustering and periodicity was fully explored and understood. The point was that if the observable is maximised at a periodic point, then,  whenever the orbit enters in a ball around this point causing the observation of an exceedance of an high level, we are forced to return to a vicinity of the periodic point and possibly observe another exceedance of the same high level. In fact, the conditional probability of avoiding a return to the initial ball around the periodic point of period $p$, given that we hit the ball $p$ iterates before, gives us the value of the extremal index. 
In \cite{AFFR16}, by considering multiple maximising points which were dynamically linked, a fake periodic behaviour is introduced which produced clustering of high observations and, consequently, an extremal index less than 1.  The linked points $\tilde\zeta\neq\hat\zeta$ were such that the observable function reached its maximum value (possibly $\infty$) at these points, and where there exists $j\in\N$ for which $f^j(\tilde \zeta)=\hat \zeta$.  
The heuristic is clear: the dynamic link $f^j(\tilde \zeta)=\hat \zeta$ forces the orbits which first hit vicinities of $\tilde \zeta$ to visit vicinities of $\hat \zeta$ after $j$ iterations and therefore ensures the appearance of clusters of high observations. This meant that understanding the recurrence properties of the maximising set was crucial to anticipate clustering patterns, which was explored later in \cite{AFFR17}, for countably many maximising points, in \cite{FFRS20, FFS21} for maximising fractal sets (like Cantor sets) and in \cite{CFF24} to study clustering profiles.     

Contrary to the univariate setting, we have now different maximising sets corresponding to the different components. Their spatial connections, given by their intersecting (or overlapping) points, and their temporal connections, which appear when they share points of the same orbit, give rise to a plethora of different limiting component extremal interdependencies that we will study. 

 In our examples the extremes are realised on sets $\cZ= \cup_{i=1}^d\cZ_i$.
We will assume that $\psi_i(x) = g_i(d(x, \cZ_i))$ for a set $\cZ_i$, where $d$ is the induced Hausdorff metric.  As usual, in order for \eqref{eq:Mmargs} to hold, the $g_i$ should each be one of the three types, see \cite[(4.2.3)--(4.2.5)]{LFFF16} for the general case, examples of each are:  
\begin{equation}
(1)\ g_i(t) = -\log t; \  (2)\ g_i(t)= t^{-1/\alpha}; \  (3) \ g_i(t)=D- t^{1/\alpha},
\label{eq:types}
\end{equation}
where $\alpha>0$ and $D\in \R$.
We set $U^{(n)}(\ttau) := \{\XX_0\not\leq \uu_n\}$ (note that in this paper we also call this set $A_n^{(0)}(\ttau)$).   Moreover, we write $U_i^{(n)}(\tau_i):=  \{\XX_{0i}>u_{ni}(\tau_i)\}$, so $U^{(n)}(\ttau) = \cup_{i=1}^d U_i^{(n)}$.

For us to be able to satisfy \eqref{eq:def_un} and \eqref{def:thetan}, we will need some regularity of our measure and our observables.  These will generally involve the scaling of the measures of the sets $U_i^{(n)}(\tau_i),$ $A_n^{(q)}(\ttau)$, and preimages of (parts of) such sets.  To keep the setup flexible, we refer to the required properties as `regularity', but note that there will be concrete examples where these will be satisfied, for example when $\PPsi$ is continuous in some $B_\eps(\zeta)\sm \{\zeta\}$ and $\mu$ has a smooth density with respect to Lebesgue at $\zeta\in \cZ$.

We also need to check that conditions  $\D(\uu_n)$ and $\D'(\uu_n)$ hold. Here, is where we take advantage of the design of the conditions, which were purposely adapted to the dynamical setting. These conditions are not much different from the ones in the multidimensional setting of \cite{FFM20,FFT25} and follow in practically the same way. 

\begin{remark}
We observe that the connection between the shapes of $g_i$ and the tail of distribution functions associated to the random variables involved and the respective domains of attraction for maxima is explained carefully in \cite[Section~4.2]{LFFF16}. In particular, depending on the specific type of tail that applies, there are natural sequences $u_{ni}$ associated with the respective linear normalisation from the classical theory of Extremes. However, we point out that all that we require from the $u_{ni}$ is that \eqref{eq:def_un} holds. This means that the particular form of the linear normalisation defining the $u_{ni}$ is not important in this setting, but instead the requirement is that such normalising sequences fulfilling  \eqref{eq:def_un} do exist. In particular, one can use the respective quantile distribuiton functions to obtain such sequences, for example.
\end{remark}

Condition $\D(\uu_n)$ follows easily from sufficiently fast decay of correlations (summable rates are enough). We refer to the discussions in \cite[Section 5.1]{F13} or \cite[Section 4.4]{LFFF16}, for non-uniformly expanding systems or to \cite[Section 2]{GHN11}, for higher dimensional systems with contracting directions.

Condition $\D'(\uu_n)$ is typically more involved because it depends a lot on the recurrence properties of the maximal set, which often requires a local analysis of the dynamics there. However, if the systems have a strong form of decay of correlations, like uniformly expanding systems do, it can actually be easily checked globally see (\cite[Section 4.2.4]{LFFF16}). 

To illustrate the preceding discussion, we state the following proposition, without aiming at the fullest generality, as it covers most of the forthcoming examples. But first, we need to introduce a few more notations and definitions.

Let $\mathcal{C} \subset L^1(\mu)$ be a Banach space of real-valued measurable functions defined on $\mathcal{X}$. We say that we have {\em summable decay of correlations} for observables in $\mathcal{C}$ against $L^1(\mu)$ if there exists a sequence $(\rho_n)_{n \ge 1}$ of positive real numbers with $\sum_{n \ge 1} \rho_n < \infty$ such that for all $\phi \in \mathcal{C}$ and all $\psi \in L^1(\mu)$,
\begin{equation*}
\left|\int \phi (\psi \circ f^n) d\mu - \int \phi d \mu \int \psi d \mu \right| \le \rho_n \|\phi \|_{\mathcal{C}} \| \psi \|_{L^1_\mu}.
\end{equation*}

For some $x\in \mathcal{X}$ and a measurable subset $A\subset \mathcal{X}$, we define the {\em first hitting time} of $x$ to $A$ as
\begin{equation*}
r_A(x) = \inf \left\{ j \ge 1 \, : \, f^j(x) \in A \right\},
\end{equation*}
and the {\em first return} from $A$ to $A$ as
\begin{equation*}
R(A) = \inf_{x \in A} r_A(x).
\end{equation*}

\begin{proposition}
\label{prop:applications}
Let $f: \mathcal X \to \mathcal X$ be a dynamical system on $\mathcal{X} = [0,1]$, preserving the probability measure $\mu$ absolutely continuous with respect to Lebesgue, with summable decay of correlations for observables in the Banach space $\mathcal{C} \subset L^1(\mu)$ against $L^1(\mu)$. 
Assume that each $\cZ_i$ is a finite collection of points such that the dynamics $f$ is continuous in a neighbourhood of each point of $\cZ_i$. We also require the density of $\mu$ exists, is continuous and lies in $(0,\infty)$ in a neighbourhood of such points.
Assume that $\psi_i(x) = g_i(d(x, \cZ_i))$ where $g_i$ is of the form specified in \eqref{eq:types} and that $\theta(\ttau)$ given by \eqref{def:thetan} is well defined. 
Assume furthermore that there exists $p \ge 0$ such that for all $\ttau \in [0,\infty)^d \setminus \mathbb 0$
\begin{equation} \label{eq:R_cond}
p = \min \left\{ j \ge 0 \, : \lim_{n \to \infty} R(A_n^{(j)}(\ttau)) = \infty \right\}    
\end{equation}
and that $\mathcal{C}$-norms of the family of indicator functions $\left( \mathds{1}_{A_n^{(p)}(\ttau)}\right)_{n \ge 1}$ are uniformly bounded.

Then the assumptions of Theorem \ref{thm:distributional-convergence} hold and in particular we have 
$$
\lim_{n\to\infty}\mu(\MM_n\leq \uu_n(\ttau))=\e^{-\theta(\ttau)\hat\Gamma(\ttau)},
$$
where $\hat\Gamma(\ttau)$ and $\theta(\ttau)$ are given by \eqref{eq:unvector} and \eqref{def:thetan} respectively.
\end{proposition}

\begin{remark} We refer to \cite{FFM18} for example for a discussion on the possible choices for the Banach space $\mathcal{C}$. We simply remark that in the case of uniformly expanding piecewise $C^2$ maps with a mixing measure absolutely continuous with respect to Lebesgue (\emph{acip}), it is enough to consider $\mathcal{C} = {\rm BV}$ the space of functions with bounded total variation on $[0,1]$. Furthermore, we notice that \eqref{eq:R_cond} is trivially satisfied for the doubling and tripling map, and that for more general maps a continuity argument and the Hartman-Grobman Theorem can be used to show the convergence if the system is continuous along the orbits of the points of $\cZ_i$. We also note that higher-dimensional uniformly expanding examples can in principle be treated, since suitable choices of the Banach space $\mathcal{C}$ are available for this setting. However, to keep the presentation as simple as possible, we chose to state Proposition \ref{prop:applications} only at the present level of generality. In the higher-dimensional case, an additional difficulty arises in verifying the existence of the limit \eqref{def:thetan} defining $\theta(\ttau)$
(see Remark \ref{rmk:measure}).
\end{remark}

\begin{remark}   \label{rmk:slow_mixing}
Moreover, adapting \cite[Section 5, Theorem 2.C]{FFM18} we can show that if the system admits an induced system for which we can check  conditions $\D(\uu_n)$ and $\D'(\uu_n)$, then the conclusion of Theorem~\ref{thm:distributional-convergence} applies to both the induced and the original system, which means that we can immediately apply our findings to slowly mixing systems (with non-summable rates), such as Manneville-Pomeau type of maps  with indifferent fixed points as well as some quadratic maps, both with acips.  
We refer to \cite[Section~4.2]{FFM20} for a list of systems for which we can apply our results.
\end{remark}

 In many of the examples in this section, we will work out the theory in the general settings just described, for example giving a formula for $\theta(\ttau)$, and then to clarify the phenomena at play and provide more concrete formulae we restrict to the specific example of the doubling map $f:x\mapsto 2x \mod 1$ or the tripling map $f:x\mapsto 3x \mod 1$ on $[0,1]$ with Lebesgue as the invariant measure $\mu$: we use the latter when either we need a fixed point where $f$ is a diffeomorphism in a 2-sided neighbourhood or when we need a map of degree greater than two.  Note that these systems satisfy the assumptions of Proposition \ref{prop:applications}. 

The distinctive features of the dependence functions obtained in each example will be analysed in light of the spatial and temporal connections between the maximising sets $\mathcal Z_i$, $i=1,\ldots d$, namely, the geometric links created by their overlaps (which occur when $\mathcal Z_i\cap\mathcal Z_j\neq\emptyset$, for $i\neq j\in\{1,\ldots,d\}$) and the temporal links created by their inter-recurrence properties (which appear when $\mathcal Z_i\cap f^{-k}(\mathcal Z_j)\neq \emptyset$, for $i\neq j\in\{1,\ldots,d\}$ and $k\in\N$), which will give us some understanding about the mechanisms responsible for the creation of componentwise extremal limiting dependence.  

 \begin{remark}
     \label{rmk:gibbs}
     We can do calculations similar to the ones below when we have a Gibbs measure $\mu$, absolutely continuous with respect to a conformal measure $m$ and such that $\frac{d\mu}{dm}(\zeta_i)\in (0, \infty)$, see \cite[Lemma 3.1]{FFT12} for the implications of this fact and \cite[Section 7.3]{FFT15} to see that the condition can be commonly satisfied.
 \end{remark}

\subsection{Points not dynamically linked}

We start by illustrating the theory with the elementary case of bivariate processes arising from observables where the components are maximised at the same point or at two distinct points with no dynamical link between them, which in particular means that the common point is not periodic. 

\subsubsection{Common non-periodic maximal point}

Let $\cZ=\{\zeta\}$ and $\psi_i(x)=g_i(\dist(x,\zeta))$, for $i=1,2$, where  $\bigcup_{j\geq 1} \left( \cZ \cap f^{-j}(\cZ)\right)=\emptyset$. Assuming that $g_i$ is as in \eqref{eq:types} and $\mu$ is sufficiently regular then for every $\ttau=(\tau_1,\tau_2)\in(0,\infty)^2$ there exists $(\uu_n(\ttau))_{n}$ such that \eqref{eq:def_un} holds.

Observe that $\{X_{0i}>u_{ni}(\tau_i)\}=B_{g_i^{-1}(u_{ni}(\tau_i))}(\zeta)$, $i=1,2$, where $B_\eps(\zeta)$ denotes a ball of radius $\eps$ around $\zeta$ and, therefore, whenever $\tau_1<\tau_2$  we must have that for all $n$ sufficiently large $\{X_{01}>u_{n1}(\tau_1)\}\subset\{X_{02}>u_{n2}(\tau_2)\}$, and vice-versa if $\tau_1>\tau_2$.
It follows that 
\begin{align*}
\hat\Gamma(\ttau)&=\lim_{n\to\infty}n\mu(\XX_0\not\leq \uu_n(\ttau))=\lim_{n\to\infty}n\mu\left(\bigcup_{i=1}^2\{X_{0i}>u_{ni}(\tau_i)\}\right)
=\lim_{n\to\infty}n\sum_{i=1}^2\mu\left(X_{0i}>u_{ni}(\tau_i)\right)\\
&\hspace{2cm}-\lim_{n\to\infty}n\mu\left(\bigcap_{i=1}^2\{X_{0i}>u_{ni}(\tau_i)\}\right)=\tau_1+\tau_2-\min\{\tau_1,\tau_2\}=\max\{\tau_1,\tau_2\}.
\end{align*}

This means that in this case we have perfect association since the lower bound in \eqref{eq:bounds} is achieved and $\hat H(\tt)=C_{\hat H}(\tt)=\min\{t_1,t_2\}$. Moreover, assuming that the dynamical system has sufficiently fast decay of correlations, we have that both conditions $\D(\uu_n(\ttau))$ and $\D'(\uu_n(\ttau))$ hold, where in fact one can show that $\Delta^{(0)}(\uu_n(\ttau))=0$, which implies that $\theta(\ttau)=1$ and therefore in this case $\Gamma(\ttau)=\hat\Gamma(\ttau)$, $G(\ttau)=\e^{-\max\{\tau_1,\tau_2\}}$ and we also have $H(\tt)=C_{H}(\tt)=\min\{t_1,t_2\}$. Note that we can also write 
$$\Gamma(\ttau)=(\tau_1+\tau_2)\left(1-\min\left\{\frac{\tau_1}{\tau_1+\tau_2},\frac{\tau_2}{\tau_1+\tau_2}\right\}\right),$$
which means that $D(\alpha)=1-\min\{\alpha,1-\alpha\}=\max\{\alpha,1-\alpha\}$.

To make our example slightly more concrete, suppose that  $\mu(B_r(\zeta)) \sim c r^d$, $g_1(t)= -\log t$ and $g_2(t) = t^{-1}$.  Then in the above argument we can take $u_{n1}(\tau_1) = \frac1d\left(\log(cn)-\log\tau_1\right)$ and  $u_{n2}(\tau_2) = \left(\frac{cn}{\tau_2}\right)^{\frac1d}$. 

Note that, in this case, there is a complete overlap between the maximal sets $\mathcal Z_1=\mathcal Z_2$, but no temporal connection exists between them, which explains the Extremal Index function being equal to 1. This combination is reflected in a stable dependence function corresponding to a perfect association between the two components, meaning that whenever we see an extreme observation in one of the components, then the other one will show the same type of behaviour.

\subsubsection{Distinct non-linked maximal points}
\label{sssec:dnlmax}

For $\zeta_1\neq \zeta_2$ let $\cZ_1=\{\zeta_1\}$, $\cZ_2=\{\zeta_2\}$ and $\psi_i(x)=g_i(\dist(x,\cZ_i))$, for $i=1,2$, where $\bigcup_{j\geq 1} \left( \cZ \cap f^{-j}(\cZ)\right)=\emptyset$, with $\cZ = \cZ_1\cup\cZ_2$. Assume that $g_i$, $\mu$ and $(\uu_n(\ttau))_{n}$ are as above.

Observe that $\{X_{0i}>u_{ni}(\tau_i)\}=B_{g_i^{-1}(u_{ni}(\tau_i))}(\zeta_i)$ for $i=1,2$ and since $\zeta_1$ and $\zeta_2$ are distinct then for every fixed $\ttau$ and $n$ sufficiently large we have $\{X_{01}>u_{n1}(\tau_1)\}\cap\{X_{02}>u_{n2}(\tau_2)\}=\emptyset$. It follows that 
\begin{align*}
\hat\Gamma(\ttau)& =\lim_{n\to\infty}n\mu(\XX_0\not\leq \uu_n(\ttau))=\lim_{n\to\infty}n\mu\left(\bigcup_{i=1}^2\{X_{0i}>u_{ni}(\tau_i)\}\right)\\
&
=\lim_{n\to\infty}n\sum_{i=1}^2\mu\left(X_{0i}>u_{ni}(\tau_i)\right)
-\lim_{n\to\infty}n\mu\left(\bigcap_{i=1}^2\{X_{0i}>u_{ni}(\tau_i)\}\right)=\tau_1+\tau_2.
\end{align*}
This means that in this case we have asymptotic extremal independence and $\hat H(\tt)=C_{\hat H}(\tt)=t_1\cdot t_2$. Moreover, assuming that the dynamical system has sufficiently fast decay of correlations, we have that condition both conditions $\D(\uu_n(\ttau))$ and $\D'(\uu_n(\ttau))$ hold, where in fact one can show that $\Delta^{(0)}(\uu_n(\ttau))=0$, which implies that $\theta(\ttau)=1$ and therefore in this case $\Gamma(\ttau)=\hat\Gamma(\ttau)$, $G(\ttau)=\e^{-(\tau_1+\tau_2)}$ and we also have $H(\tt)=C_{H}(\tt)=t_1\cdot t_2$. Clearly, $D(\alpha)=1$, in this case.

Observe that in this case there is no overlap or any time connection between the maximal sets $\mathcal Z_1=\mathcal Z_2$, which translates into complete independence regarding the extremal limiting behaviour of the two components. This is reflected by an Extremal Index function and Pickands dependence function both equal to $1$. In other words, the observation of an extreme value in one component is asymptotically independent of the occurrence of an extreme value on the other component. 

\subsection{Distinct linked points}

In this case we will build up a lot of the notation and machinery also required for later cases.  Assume $\cZ_1=\{\zeta\}$ and $\cZ_2=\{f(\zeta)\}$.  

\subsubsection{Non-periodic case}
\label{subsec:non-periodic-case-no-overlap}

Assume first that $\{\zeta, f(\zeta)\} \cap \{\cup_{n\ge 2} f^n(\zeta)\} = \es$.  So in particular, given $q\in \N$ and $\ttau=(\tau_1, \tau_2)$, for sufficiently large $n$, $x\in U_2^{(n)}(\tau_2)$ implies 
\begin{equation}
\{f(x), \ldots, f^q(x)\}\cap  \left(U_1^{(n)}(\tau_1)\cup  U_2^{(n)}(\tau_2)\right)=\es \implies A_n^{(q)}(\ttau) = A_n^{(1)}(\ttau).
\label{eq:dlA1}
\end{equation}

Given $\alpha\in [0, 1]$, define 
$$\theta_\zeta(\alpha):= \lim_{n\to \infty} \frac{\mu\left(U_1^{(n)}(\alpha)\sm f^{-1}U_2^{(n)}(1-\alpha)\right)}{\mu\left(U_1^{(n)}(\alpha)\right)},$$
and observe that, since the measure $\mu$ is absolutely continuous with respect to Lebesgue, with a one-dimensional density continuous at the points of $\cZ_i$, by regulatity, we have, for $\alpha= \frac{\tau_1}{\tau_1+\tau_2}$, 
\begin{equation} \label{eq:theta_regular}
    \theta_\zeta(\alpha):= \lim_{n\to \infty} \frac{\mu\left(U_1^{(n)}(\tau_1)\sm f^{-1}U_2^{(n)}(\tau_2)\right)}{\mu\left(U_1^{(n)}(\tau_1)\right)},
\end{equation}
so $\theta_\zeta(\alpha)$ represents the asymptotic probability, given $x\in U_1^{(n)}(\tau_1)$, that $f(x)\notin U_2^{(n)}(\tau_2)$, and since $n$ can be assumed large, that $f(x)\notin U_1^{(n)}(\tau_1)\cup U_2^{(n)}(\tau_2)$.  
We will use this type of idea throughout our examples.

 Then $\XX_0=\XX_0(x)\nleq \uu_n(\tau_1, \tau_2)$ means that $x\in U_1^{(n)}(\tau_1)\cup U_2^{(n)}(\tau_2)$, which implies, for all large $n$,

\begin{enumerate}
\item since 
$$\mu\left(X_{01}>u_{n1}(\tau_1)|\XX_0(x)\nleq \uu_n(\tau_1, \tau_2)\right)=\frac{\mu\left(U_1^{n})(\tau_1)\right)}{\mu\left(U_1^{n})(\tau_1)\cup U_n^{(2)}(\tau_2)\right)}\sim \frac{\tau_1}{\tau_1+\tau_2}=\alpha,$$  with asymptotic probability $\alpha$ we have $X_{01}>u_{n1}(\tau_1)$ (i.e., $x\in U_1^{(n)}(\tau_1)$).  Then $X_{11}< u_{n1}(\tau_1)$ with probability 1, but $X_{12}< u_{n2}(\tau_2)$ (i.e., $x \in U_1^{(n)}(\tau_1)\sm f^{-1}U_2^{(n)}(\tau_2)$) with probability $\theta_\zeta(\alpha)$;
\item  with asymptotic probability $1-\alpha$ we have $X_{02}>u_{n2}(\tau_2)$ (i.e., $x\in U_2^{(n)}(\tau_2)$).  Then $X_{11}< u_{n1}(\tau_1)$ and $X_{12}< u_{n2}(\tau_2)$ with probability 1 (since $f(x) \notin U_2^{(n)}(\tau_2)\cup U_1^{(n)}(\tau_1)$).
\end{enumerate}
As in \eqref{eq:dlA1}, there are no other entries of importance up to time $q$.
So summing the probabilities gives
$$\theta(\alpha, 1-\alpha)= \alpha\theta_\zeta(\alpha)+ 1-\alpha
.$$

To find $\theta_\zeta(\alpha)$, we observe that $U_1^{(n)}(\alpha)= B_{g_1^{-1}(u_{n1}(\alpha))}(\zeta)$ and $U_2^{(n)}(1-\alpha)= B_{g_2^{-1}(u_{n2}(1-\alpha))}(f(\zeta))$.  Moreover, if $f$ is conformal then $f^{-1}U_2^{(n)}(1-\alpha)$ is an approximate ball of radius $\frac1{|Df(\zeta)|}g_2^{-1}(u_{n2}(1-\alpha))$, so we are left to find the relative (to $U_1^{(n)}(\alpha)$) measure of the annulus
$$A_{\frac1{|Df(\zeta)|} g_2^{-1}(u_{n2}(1-\alpha)), g_1^{-1}(u_{n1}(\alpha)},$$
assuming that the inner radius is strictly smaller than the outer radius so that this makes sense.
In a setting where $\mu$ is an acip with density $\rho$, $f$ is conformal  and the observables are all of the same form, for $c_d$ the volume of the $d$-dimensional unit ball, the measure of our annulus is asymptotically $\frac1{c_d}\left(\frac{\alpha}{n\rho(\zeta)}- \frac1{|Df(\zeta)|}\frac{1-\alpha}{n\rho(f(\zeta))}\right)$ (we are assuming that $\frac1{|Df(\zeta)|}\frac{1-\alpha}{c_d n\rho(f(\zeta))}< \frac{\alpha}{c_d n\rho(\zeta)}$, if not then $\theta_\zeta(\alpha, 1-\alpha)=0$).  Since $\mu(U_1^{(n)}(\alpha))\sim \frac\alpha{c_dn\rho(\zeta)}$, 
hence,
$$\theta_\zeta(\alpha)= 
\max\left\{0,  1- \frac{\rho(\zeta)}{\rho(f(\zeta))}\frac{1-\alpha}{\alpha}\frac1{|Df(\zeta)|}\right\}.$$

Thus 
 $$\theta\left(\alpha, 1-\alpha\right)=\alpha\max\left\{0, 1- \frac{\rho(\zeta)}{\rho(f(\zeta))}\frac{1-\alpha}{\alpha}\frac1{|Df(\zeta)|}\right\}+ 1-\alpha.
 $$
In the doubling map case with Lebesgue measure, this becomes
\[\theta(\alpha, 1-\alpha)= \begin{cases} 1-\alpha & \text{ if } \alpha \le \frac13,\\
\frac{1+\alpha}2 & \text{ if } \alpha > \frac13,
\end{cases}
\qquad
G(\tau_1, \tau_2)= \begin{cases} e^{-\tau_2} & \text{ if } \tau_1 \le \frac{\tau_2}2,\\
e^{-\left(\tau_1+\frac{\tau_2}2\right)} & \text{ if }\tau_1 > \frac{\tau_2}2.
\end{cases}
\]

Where $G$ is obtained by adding in the $\hat H$ term from Section~\ref{sssec:dnlmax}.  Moreover, we see here that $\theta_1=\theta_2=1$, so $\Gamma(\ttau)= \theta(\ttau)\hat\Gamma(\ttau)$ and $D(\alpha) = \Gamma(\alpha, 1-\alpha)= \theta(\alpha, 1-\alpha)$, see Figure~\ref{fig:examples22}. 

Observe that in this case there is no spatial link because the two maximal sets have an empty intersection which, in particular, implies that $\hat\Gamma(\ttau)=\tau_1+\tau_2$, corresponding to the independent case. However, there is now a temporal link between the sets since $\mathcal Z_1\cap f^{-1}(\mathcal Z_2)\neq \emptyset$. This link creates a clustering profile which is captured by the Extremal Index function which is not constant equal to $1$ anymore (although the componentwise Extremal Indices $\theta_1, \theta_2$ are still equal to $1$). Hence, we have cross inter-recurrences but not intra-recurrences. The interesting point is that the stable dependence function $\Gamma(\ttau)$ and the respective Pickands dependence function show that the temporal link also creates dependence between the components' extremal behaviour. Also, note that Pickands dependence function is asymmetric reflecting the fact that there is temporal link taking the orbit from close to $\mathcal Z_1$ to close to $\mathcal Z_2$, but not the other way around. In fact, the clustering profile responsible for the asymmetric Pickands dependence function tells us that an extreme value on the first component is likely to be followed by another high observation in the second component, but the occurrence of an extreme value on the second component is not followed by any high observation in any of the components.

\begin{remark}
    \label{rmk:measure} 
    In order to guarantee that \eqref{eq:theta_regular} holds, we used that the measure $\mu$ is absolutely continuous with respect to the one-dimensional Lebesgue measure on an interval and has a density which is positive, finite and continuous in a neighbourhood of the points of $\cZ_i$. For more general higher-dimensional systems, say in the two dimensional case, one should impose conditions along the following lines. Given $A \subset \mathcal{X}$, $x \in \mathcal{X}$ and $\lambda > 0$, define $$(\lambda,x)(A) = \left\{x + ty \in \mathcal{X}\, : \, y \in A \quad {\rm and } \quad t = \lambda \| y -x \| \right\}.$$
    We require for any $A, B$ subsets of $\mathcal{X}$ with boundaries consisting of finite pieces of Jordan curves, for $c > 0$,
    $$\lim_{\lambda \to 0} \frac{\left( \frac{\mu((c \lambda, x)(A) \setminus (c \lambda, x)(B))}{\mu((c \lambda, x)(A))}\right)}{ \left( \frac{\mu((\lambda, x)(A) \setminus (\lambda,x)(B))}{ \mu((\lambda, x)(A))} \right)} = 1,$$
    and, for $A$ as above, for all $\epsilon > 0$ there exists $\delta > 0$ such that for all small $\lambda > 0$,
    $$ \left| \left( \frac{\mu((\lambda(1-\delta), x)(A) \setminus (\lambda, x)(A))}{\mu((\lambda,x)(A))}\right) -1\right| < \epsilon.$$
\end{remark}

\subsubsection{Periodic case}
\label{ssec:dlp_per}

Next we assume that $\cZ_1=\{\zeta\}$, $\cZ_2=\{f(\zeta)\}$ with $f^2(\zeta)=\zeta$ (and $f(\zeta)\neq \zeta$).  As we will see below, in contrast to \eqref{eq:dlA1}, here we need only consider $A_n^{(2)}(\ttau)$.  Then $\XX_0\nleq \uu_n(\tau_1, \tau_2)$ implies, for all large $n$ and $\alpha$ as above,

\begin{enumerate}
\item with asymptotic probability $\alpha$ we have $X_{01}>u_{n1}(\tau_1)$ (i.e., $x\in U_1^{(n)}(\tau_1)$).  Then $X_{11}< u_{n1}(\tau_1)$ with probability 1 (since $f(x)\notin U_1^{(n)}(\tau_1)$), but $X_{12}< u_{n2}(\tau_2)$ only if $x\in U_1^{(n)}(\tau_1)\sm f^{-1}U_2^{(n)}(\tau_2)$; then $X_{21}<u_{n1}(\tau_1)$ only if $x\in U_1^{(n)}(\tau_1)\sm f^{-2}U_1^{(n)}(\tau_1)$; in total to evaluate the probability of the event $\{\XX_1<\uu_n(\ttau)| X_{01}>u_{n1}(\tau_1)\}$  we require the relative probability that $x\in U_1^{(n)}(\tau_1)\sm \left(f^{-1}U_2^{(n)}(\tau_2)\cup f^{-2}U_1^{(n)}(\tau_1)\right)$, i.e., the ratio of the measure of this set compared to the measure of $U_1^{(n)}(\tau_1)$;

\item  with asymptotic probability $1-\alpha$ we have $X_{02}>u_{n2}(\tau_2)$ (i.e., $x\in U_2^{(n)}(\tau_2)$).  Then $X_{12}< u_{n2}(\tau_2)$ with probability 1 (since $f(x)\notin U_2^{(n)}(\tau_2)$), but $X_{11}< u_{n1}(\tau_1)$ only if $x\in U_2^{(n)}(\tau_2)\sm f^{-1}U_1^{(n)}(\tau_1)$; then $X_{22}<u_{n2}(\tau_2)$ only if $x\in U_2^{(n)}(\tau_2)\sm f^{-2}U_2^{(n)}(\tau_2)$; in total to evaluate the probability of the event $\{\XX_1<\uu_n(\ttau)| X_{02}>u_{n2}(\tau_2)\}$ we require the relative probability that  $x\in U_2^{(n)}(\tau_2)\sm \left(f^{-1}U_1^{(n)}(\tau_1)\cup f^{-2}U_2^{(n)}(\tau_2)\right)$,  i.e., the ratio of the measure of this set compared to the measure of $ U_2^{(n)}(\tau_2)$.
\end{enumerate}

We can see that for given $q\in \N$, for any large $n$, $A_n^{(q)}(\ttau)= A_n^{(2)}(\ttau)$.
So setting 
$$\theta_\zeta(\alpha):= \lim_{n\to \infty}\frac{\mu\left(U_1^{(n)}(\alpha)\sm \left(f^{-1}U_2^{(n)}(1-\alpha)\cup f^{-2}U_1^{(n)}(\alpha)\right)\right)}{\mu\left(U_1^{(n)}(\alpha)\right)}$$
$$\theta_{f(\zeta)}(\alpha):= \lim_{n\to \infty} \frac{\mu\left(U_2^{(n)}(1-\alpha)\sm \left(f^{-1}U_1^{(n)}(\alpha)\cup f^{-2}U_2^{(n)}(1-\alpha)\right)\right)}{\mu\left(U_2^{(n)}(1-\alpha)\right)},$$
we obtain
$$\theta(\alpha, 1-\alpha) = \alpha \theta_\zeta(\alpha)+ (1- \alpha) \theta_{f(\zeta)}(\alpha).$$

To compute the annuli here is more involved since we need to incorporate the relative strengths of the derivatives $Df(\zeta)$ and $Df^2(\zeta)$ for $\theta_\zeta$ and  $Df(f(\zeta))$ and $Df^2(f(\zeta))$ for $\theta_{f(\zeta)}$, as well as the relative sizes of $\alpha$ and $1-\alpha$.  If $f$ is conformal, for $\theta_\zeta(\alpha, 1-\alpha)$ we consider the relative size of the annulus Lebesgue measure
$$\frac\alpha{n\rho(\zeta)} - \max\left\{\frac1{|Df(\zeta)|}\frac{1-\alpha}{n\rho(f(\zeta))}, \frac1{|Df^2(\zeta)|}\frac{\alpha}{n\rho(\zeta)}\right\},$$
which gives
$$\theta_\zeta(\alpha)=\max\left\{0,  1 - \max\left\{\frac1{|Df(\zeta)|}\frac{1-\alpha}\alpha \frac{\rho(\zeta)}{\rho(f(\zeta))}, \frac1{|Df^2(\zeta)|}\right\}\right\}$$
Similarly, 
$$\theta_{f(\zeta)}(\alpha)= \max\left\{0, 1 - \max\left\{\frac1{|Df(f(\zeta))|}\frac\alpha{1-\alpha} \frac{\rho(f(\zeta))}{\rho(\zeta)}, \frac1{|Df^2(\zeta)|}\right\}\right\}.$$
Therefore in this case,
\begin{align*} \theta(\alpha, 1-\alpha) &= \alpha\max\left\{0,  1 - \max\left\{\frac1{|Df(\zeta)|}\frac{1-\alpha}\alpha \frac{\rho(\zeta)}{\rho(f(\zeta))}, \frac1{|Df^2(\zeta)|}\right\}\right\}\\
&\qquad +(1-\alpha)  \max\left\{0, 1 - \max\left\{\frac1{|Df(f(\zeta))|}\frac\alpha{1-\alpha} \frac{\rho(f(\zeta))}{\rho(\zeta)}, \frac1{|Df^2(\zeta)|}\right\}\right\}.
\end{align*}
In the doubling map case with Lebesgue measure (here $\zeta=1/3$), we compute
\[  \theta(\alpha, 1-\alpha)= \begin{cases} \frac34(1-\alpha) & \text{ if } \alpha \le \frac13,\\
\frac12 & \text{ if } \alpha \in (\frac13, \frac23],\\
\frac{3\alpha}4 & \text{ if } \alpha > \frac23,
\end{cases}
\qquad
G(\tau_1, \tau_2)= \begin{cases} e^{-\frac34\tau_2} & \text{ if }  \tau_1 \le \frac{\tau_2}2,\\
e^{-\frac12(\tau_1+\tau_2)} & \text{ if } \tau_1 \in (\frac{\tau_2}2, 2\tau_2],\\
e^{-\frac{3\tau_1}4} & \text{ if } \tau_1 > 2\tau_2.
\end{cases}
\]

Where we obtained $G$ by using the $\hat H$ term from Section~\ref{sssec:dnlmax} (recall that $\hat\Gamma(\ttau)=\tau_1+\tau_2$).  Here we compute $\theta_1=\theta_2=3/4$, so $\Gamma(\ttau)=\frac43 \theta(\ttau)\hat \Gamma(\ttau)$. As usual 
$D(\alpha) = \Gamma(\alpha, 1-\alpha)$, giving 
$$D(\alpha)= \begin{cases} 
1-\alpha & \text{ if } \alpha \in [0,1/3],\\
\frac23 & \text{ if } \alpha \in (1/3, 2/3],\\
\alpha & \text{ if } \alpha \in (2/3, 1].
\end{cases} $$ 
See Figure~\ref{fig:examples22}.

As in the previous subsection, in this case there is no spatial connection between the maximal sets and therefore $\hat\Gamma(\ttau)$ is the same and corresponds to the independent case. As before, there is a time link between the two maximal sets, but the main difference is that, now, this link is symmetric in the sense that if the orbit is close to $\mathcal Z_1$, then it will be close $\mathcal Z_2$ in the following iterate and vice-versa, due to periodicity. The Extremal Index function is now symmetric which in turn implies that $\Gamma(\ttau)$ is symmetric, as well as the corresponding Pickands dependence function. The clustering profile in this case is as follows: a high observation in one component is likely to be followed by another one in the other component, which is then followed by another high value in the first component and the cycle repeats itself. However, note that the high values in each component fade away with time due to the expansion at the periodic orbit, which pushes the points farther and farther away. 

\begin{figure}
\includegraphics[scale=0.4]{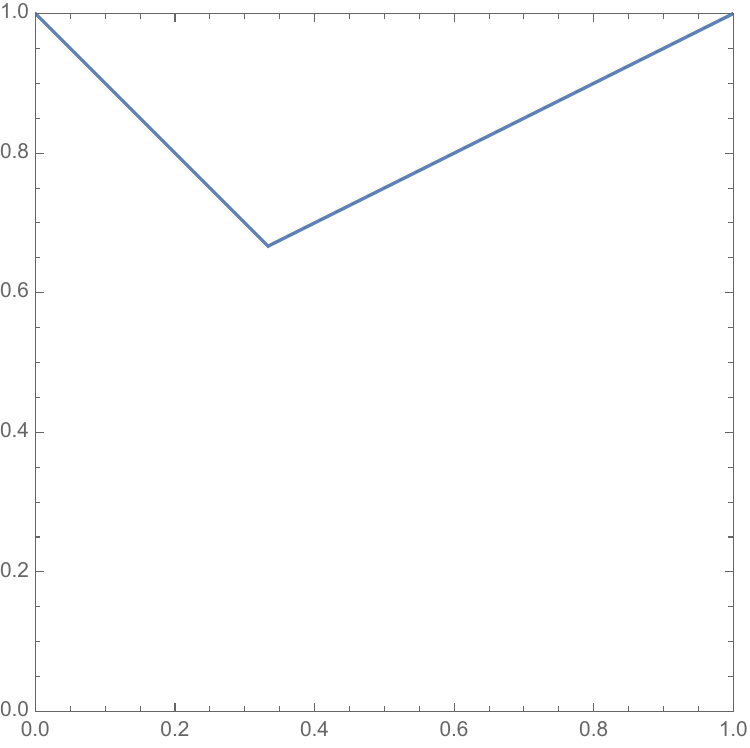}\hspace{2cm}\includegraphics[scale=0.4]{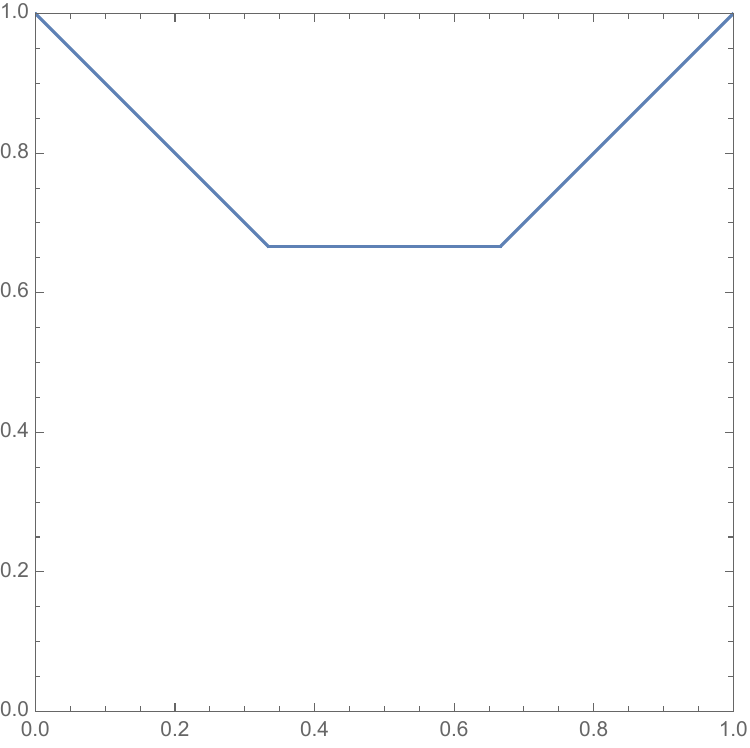}
\caption{Graphs of the Pickands dependence functions of the examples in Section~\ref{subsec:non-periodic-case-no-overlap} on the left, and Section~\ref{ssec:dlp_per}, on the right.}
\label{fig:examples22}
\end{figure}

\subsubsection{Another periodic case}

In order to give an example where $\theta_1\neq \theta_2$, we assume now that $\cZ_1=\{\zeta\}$, $\cZ_2=\{f(\zeta)\}$ with $f^2(\zeta)=f(\zeta)$ (and $f(\zeta)\neq \zeta$).  Here $\theta_\zeta(\alpha)$ is as in Section~\ref{subsec:non-periodic-case-no-overlap}, but we also need a $\theta_{f(\zeta)}(\alpha)$ to derive $\theta(\alpha, 1-\alpha) = \alpha\theta_\zeta(\alpha)+(1-\alpha)\theta_{f(\zeta)}$.  For an explicit formula, in order to have $f$ being locally diffeomorphic around each of $\zeta$ and $f(\zeta)$ we will use the tripling map (so eg $\zeta=1/6$, $f(\zeta)=1/2$).  We will not give the full details, but we see that $\theta_\zeta(\alpha)= \max\left\{0, 1-\frac13\frac{1-\alpha}\alpha\right\}$ and $\theta_{f(\zeta)}(\alpha)\equiv2/3$, so 
\[\theta(\alpha, 1-\alpha)= \begin{cases} \frac23(1-\alpha) & \text{ if } \alpha \le \frac14,\\
\frac13+\frac23\alpha & \text{ if } \alpha > \frac14,
\end{cases}
\qquad
G(\alpha, 1-\alpha)= \begin{cases} e^{-\frac23\tau_2} & \text{ if } \tau_1 \le \frac{\tau_2}3,\\
e^{-\left(\tau_1+\frac13\tau_2\right)} & \text{ if } \tau_1 > \frac{\tau_2}3.
\end{cases}
\]
Since $\theta_1=1$ and $\theta_2=2/3$, we have $\Gamma(\tau_1, \tau_2) = \theta\left(\tau_1, \frac32\tau_2\right)\hat\Gamma\left(\tau_1, \frac32\tau_2\right)$, i.e., 
\[\Gamma(\tau_1, \tau_2)= \begin{cases} \tau_2 & \text{ if } \tau_2\ge 2\tau_1,\\
\tau_1+\frac12\tau_2 & \text{ if } \tau_2< 2\tau_1,
\end{cases} 
\qquad D(\alpha) =  \begin{cases} 1-\alpha & \text{ if } \alpha\le 1/3,\\
\frac{1+\alpha}2 & \text{ if } \alpha> 1/3.\end{cases}
\]

\subsection{Overlapping points}

Here we look at three points $\{\zeta_1, \zeta_2, \zeta_3\}$, then consider  $\cZ_1=\{\zeta_1, \zeta_3\}$ and $\cZ_2=\{\zeta_2, \zeta_3\}$.

\subsubsection{Spatial dependence}\label{subsec:spatial-overlapping}
We first compute $\hat\Gamma$.
Let, for $i=1,2$, 
\[
\psi_i(x) = g_i(d(x, \cZ_i)), \: x \in \mathcal X,
\]
where $g_i$ is as in \eqref{eq:types}.
For $\ttau = (\tau_1, \tau_2) \in (0,\infty)^2$, choose $u_{n i}(\tau_i) > 0$ such that
\[
\mu(U_i^{(n)}(\tau_i)) \sim \frac{\tau_i}{n},
\]
where as before $U_i^{(n)}(\tau_i) = \left\{ \psi_i > u_{n i}(\tau_i) \right\}$.

We assume that for all $\ttau$ and all $n$ large enough, $U_i^{(n)}(\tau_i)$ can be written as a disjoint union
\[
U_i^{(n)}(\tau_i) = V_i^{(n)}(\tau_i) \cup \hat V_i^{(n)}(\tau_i),
\]
where $V_i^{(n)}(\tau_i)$ (resp. $\hat V_i^{(n)}(\tau_i)$) is a neighbourhood of $\zeta_i$ (resp. $\zeta_3$), and that
\[
 \mu\left(V_i^{(n)}(\tau_i)\right) \sim p_i \mu\left(U_i^{(n)}(\tau_i)\right),
\]
for some $p_i \in (0,1)$, so that 
\[
 \mu\left(V_i^{(n)}(\tau_i)\right) \sim \frac{p_i \tau_i}{n} \text{ and }  \mu\left(\hat V_i^{(n)}(\tau_i)\right) \sim \frac{(1-p_i)\tau_i}{n}.
\]

We define $\hat V^{(n)}(\ttau) = \cap_{i=1}^2 \hat V_i^{(n)}(\tau_i)$ and we assume that for some $q_1(\ttau) \in (0,1)$,
\[
 \mu\left(\hat V^{(n)}(\ttau)\right) \sim q_1(\ttau) \mu\left(\hat V_1^{(n)}(\tau_1)\right).
\]
This implies that $\mu\left(\hat V^{(n)}(\ttau)\right) \sim q_2(\ttau) \mu\left(\hat V_2^{(n)}(\tau_2)\right)$ where $q_2(\ttau) = q_1(\ttau) \frac{(1-p_1)\tau_1}{(1-p_2)\tau_2}$ and 
\[
\mu\left(\hat V^{(n)}(\ttau)\right) \sim \frac{q_1(\ttau)(1-p_1) \tau_1}{n} = \frac{q_2(\ttau) (1-p_2) \tau_2}{n}.
\]

It follows that 
\begin{align*}
\hat\Gamma(\ttau)&=\lim_{n\to\infty}n\mu(\XX_0\not\leq \uu_n(\ttau))=\lim_{n\to\infty}n\mu\left(\bigcup_{i=1}^2\{X_{0i}>u_{ni}(\tau_i)\}\right)
=\lim_{n\to\infty}n\sum_{i=1}^2\mu\left(X_{0i}>u_{ni}(\tau_i)\right)\\
&\quad-\lim_{n\to\infty}n\mu\left(\bigcap_{i=1}^2\{X_{0i}>u_{ni}(\tau_i)\}\right)=\tau_1 + \tau_2 - q_1(\ttau)(1-p_1)\tau_1 = \tau_1 + \tau_2 - q_2(\ttau)(1-p_2)\tau_2.
\end{align*}

 As an example, suppose that $\mu(B_r(\zeta_k)) \sim c_k r^d$ for $k=1, 2, 3$ and that the functions $g_i$ are given by $g_i(t) = - \log t$. Then we can choose
\[
u_{n i}(\tau_i) = - \frac{1}{d} \left( \log \tau_i + \log \frac{p_i}{c_i} - \log n\right),
\]
with $p_i = \frac{c_i}{c_i+ c_3}$.  Independently of the types of $g_i$ we obtain
\begin{equation}
\label{eq:Gamma-hat-3.3}
 \hat\Gamma(\ttau) = \tau_1 + \tau_2 - c_3 \min_{i=1,2} \frac{\tau_i}{c_i + c_3}.
\end{equation}
When $\mu$ is the Lebesgue measure on $\mathcal X = [0,1]$ and all the $\zeta_i$ belong to $(0,1)$, we get 
\[
\hat\Gamma(\ttau) = \tau_1 + \tau_2 - \frac 1 2 \min_{i=1, 2} \tau_i.
\]

\subsubsection{Non-periodic case}
\label{ssec:overlap_nonper}

Suppose that $\zeta_3=f(\zeta_1)=f(\zeta_2)$ and that $\left\{\cup_{n\ge 1}f^n(\zeta_3)\right\}\cap \{\zeta_1, \zeta_2, \zeta_3\}=\es$.  The computations below show that it suffices to consider $q=1$.

We require some notation for this.  First let $U_1^{(n)}(\tau_1) = V_1^{(n)}(\tau_1)\cup \hat V_1^{(n)}(\tau_1)$ where $V_1^{(n)}(\tau_1)$ is the corresponding neighbourhood of $\zeta_1$ and $\hat V_1^{(n)}(\tau_1)$ is that of $\zeta_3$.   Similarly write $U_2^{(n)}(\tau_2) = V_2^{(n)}(\tau_2)\cup \hat V_2^{(n)}(\tau_2)$.

Now we define, for $\alpha\in (0, 1)$,
$$p_1(\alpha):= \lim_{n\to \infty} \frac{\mu\left(V_1^{(n)}(\alpha)\right)}{\mu\left(U_1^{(n)}(\alpha)\cup U_2^{(n)}(1-\alpha)\right)}, \quad p_2(\alpha):= \lim_{n\to \infty} \frac{\mu\left(V_2^{(n)}(1-\alpha)\right)}{\mu\left(U_1^{(n)}(\alpha)\cup U_2^{(n)}(1-\alpha)\right)},$$
$$p_3(\alpha):= \lim_{n\to \infty} \frac{\mu\left(\hat V_1^{(n)}(\alpha)\cup \hat V_2^{(n)}(1-\alpha) \right)}{\mu\left(U_1^{(n)}(\alpha)\cup U_2^{(n)}(1-\alpha)\right)}.$$
Note that $p_3(\alpha)=1-p_1(\alpha)-p_2(\alpha)$.
Moreover, set
$$\theta_{1}(\alpha):= \lim_{n\to \infty} \frac{\mu\left(V_1^{(n)}(\alpha)\sm f^{-1}\left(\hat V_1^{(n)}(\alpha)\cup \hat V_2^{(n)}(1-\alpha)\right)\right)}{\mu\left(V_1^{(n)}(\alpha)\right)},$$
$$\theta_{2}(\alpha):= \lim_{n\to \infty} \frac{\mu\left(V_2^{(n)}(1-\alpha)\sm f^{-1}\left(\hat V_1^{(n)}(\alpha)\cup \hat V_2^{(n)}(1-\alpha)\right)\right)}{\mu\left(V_2^{(n)}(1-\alpha)\right)}.$$

Then $\XX_0\nleq \uu_n(\tau_1, \tau_2)$ implies, for all large $n$ and $\alpha$ as above, 
\begin{enumerate}
\item with asymptotic probability $p_1(\alpha)$ we have $x\in V_1^{(n)}(\tau_1)$, in which case $X_{12}<u_{n2}(\tau_2)$ and $X_{11}<u_{n1}(\tau_1)$ both occur simultaneously with asymptotic probability $\theta_{1}(\alpha)$;
\item with asymptotic probability $p_2(\alpha)$ we have $x\in V_2^{(n)}(\tau_2)$, in which case $X_{11}<u_{n1}(\tau_1)$ and $X_{12}<u_{n2}(\tau_2)$ both occur simultaneously with asymptotic probability $\theta_{2}(\alpha)$;
\item with asymptotic probability $p_3(\alpha)$ we have $x\in \hat V_1^{(n)}(\tau_1) \cup \hat V_2^{(n)}(\tau_2)$, in which case $X_{11}<u_{n1}(\tau_1)$  and $X_{12}<u_{n2}(\tau_2)$ with probability 1.
\end{enumerate}
Summing these possibilities we obtain
\begin{align*}
\theta(\alpha, 1-\alpha)&= p_1(\alpha)\theta_{1}(\alpha)+ p_2(\alpha)\theta_{2}(\alpha)+ 1-p_1(\alpha)-p_2(\alpha)\\
& =1-p_1(\alpha)\left(1-\theta_{1}(\alpha)\right)- p_2\left(\alpha)(1-\theta_{2}(\alpha)\right).
\end{align*}

To obtain a concrete formula, again assume that we have an acip with density $\rho$, that $f$ is conformal and the observables are all of the same form.  
Define 

$$r_1:= \lim_{n\to \infty}\frac{\mu\left(\hat V_1^{(n)}(1)\right)}{\mu\left(V_1^{(n)}(1)\right)}, \ r_2:= \lim_{n\to \infty}\frac{\mu\left(\hat V_2^{(n)}(1)\right)}{\mu\left(V_2^{(n)}(1)\right)}$$
(note that by regularity we could choose any positive constant in place of 1 here)
and 
$$R(\alpha) = \lim_{n\to \infty} \frac{\mu\left(V_1^{(n)}(\alpha)\right)}{\mu\left(V_2^{(n)}(1-\alpha)\right)}.$$
Then
$$\theta_{1}(\alpha)= \max\left\{0, 1-\frac{1}{|Df(\zeta_1)|}\max\left\{\frac{\rho(\zeta_1)}{\rho(f(\zeta_1))}r_1, \frac{\rho(\zeta_1)}{\rho(f(\zeta_1))}\frac{r_2}{R(\alpha)}\right\}\right\}$$
and
$$\theta_{2}(\alpha)= \max\left\{0, 1-\frac{1}{|Df(\zeta_2)|}\max\left\{\frac{\rho(\zeta_2)}{\rho(f(\zeta_1))}r_2, \frac{\rho(\zeta_2)}{\rho(f(\zeta_2))}{r_1}{R(\alpha)}\right\}\right\}.$$
To get a more concrete idea of what is happening here, we consider the case of the doubling map with Lebesgue (in which case the measures of $V_i^{(n)}(\lambda)$ and $\hat V_i^{(n)}(\lambda)$ are the same for all $\lambda>0$, so $r_1=r_2=1$)   where we get
$$\theta_{1}(\alpha)= \max\left\{0, 1-\frac{1}{2}\max\left\{1, \frac{1-\alpha}{\alpha}\right\}\right\}, \ \theta_{2}(\alpha)= \max\left\{0, 1-\frac{1}{2}\max\left\{1, \frac{\alpha}{1-\alpha}\right\}\right\}.$$

In this case,
 $p_1(\alpha) = \frac{\alpha}{1+\max\{\alpha, 1-\alpha\}}$ and $p_2(\alpha) = \frac{1-\alpha}{1+\max\{\alpha, 1-\alpha\}}$, so \begin{align*}
\theta(\alpha, 1-\alpha) & = 1-\frac\alpha{1+\max\{\alpha, 1-\alpha\}}\left(1- \max\left\{0, 1-\frac{1}{2}\max\left\{1, \frac{1-\alpha}{\alpha}\right\}\right\}\right)\\
&\hspace{2cm} -\frac{1-\alpha}{1+\max\{\alpha, 1-\alpha\}}\left(1-\max\left\{0, 1-\frac{1}{2}\max\left\{1, \frac{\alpha}{1-\alpha}\right\}\right\}\right).
\end{align*}
Hence we can compute
\[  \theta(\alpha, 1-\alpha)= \begin{cases} \frac{3-3\alpha}{4-2\alpha} & \text{ if } \alpha \in [0,1/3],\\
\frac1{2-\alpha} & \text{ if } \alpha \in (1/3, 1/2],\\
\frac1{1+\alpha} & \text{ if } \alpha \in (1/2 ,2/3],\\
\frac{3\alpha}{2+2\alpha} & \text{ if } \alpha \in (2/3, 1],
\end{cases}
\qquad 
D(\alpha)= \begin{cases} 
1-\alpha & \text{ if } \alpha \in [0,1/3],\\
\frac23 & \text{ if } \alpha \in (1/3, 2/3],\\
\alpha & \text{ if } \alpha \in (2/3, 1].
\end{cases} 
\]
Where we used $G(\ttau) = e^{-\theta(\ttau)\left(\tau_1+\tau_2-\frac12\min\{\tau_1, \tau_2\}\right)}$ 
 with $\hat\Gamma$ from Section~\ref{subsec:spatial-overlapping}. 
In this case $\theta_1=\theta_2= 3/4$, so $\Gamma(\ttau) = \frac43 \theta(\ttau)\hat\Gamma(\ttau)=  \frac43 \theta(\ttau)\left(\tau_1+\tau_2-\frac12\min\{\tau_1, \tau_2\}\right)$.  Note that, in this case, the Pickands dependence function coincides with that of Section~\ref{ssec:dlp_per}, see Figure~\ref{fig:examples22}.

Observe that in this case we have both spatial and temporal connections between the maximal sets $\mathcal Z_1$ and  $\mathcal Z_2$. Namely, their intersection $\mathcal Z_1\cap\mathcal Z_2\neq \emptyset$ is not empty and, moreover, $\mathcal Z_1\cap f^{-1}\left(\mathcal Z_2\right)\neq \emptyset\neq f^{-1}\left(\mathcal Z_1\right)\cap\mathcal Z_2$. Both these connections have a contribution for stable and Pickands dependence functions. The clustering profile in this case gives us that an extreme observation in one of the components is likely to be followed by extreme observations on both components at the same time, which means that there exists both intra-recurrence and inter-recurrence regarding the extremal behaviour of the components. Note that in dealing with the doubling map the two preimages of $\zeta_3$ are precisely $\zeta_1$ and $\zeta_2$, which means that  simultaneous high observations in both components (corresponding to an entrance in a neighbourhood of $\zeta_3$) must have been preceded by an high observation in only one of the components (an entrance either in a neighbourhood of $\zeta_1$ or $\zeta_2$).

\subsubsection{A trivariate version}
\label{sssec:triv}

Assume the same dynamical setup as that above, but with three observables, at $\zeta_1$, $\zeta_2$ and $f(\zeta_1)=f(\zeta_2)$ respectively.  Now for $\alpha\in [0, 1]$, $\beta\in [0, 1-\alpha]$, define
$$\theta_1(\alpha, \beta):= \lim_{n\to \infty}\frac{\mu\left(U_1^{(n)}(\alpha)\sm f^{-1}U_3(1-\alpha-\beta)\right)}{\mu\left(U_1^{(n)}(\alpha)\right)},$$
$$\theta_2(\alpha, \beta):= \lim_{n\to \infty}\frac{\mu\left(U_2^{(n)}(\alpha)\sm f^{-1}U_3(1-\alpha-\beta)\right)}{\mu\left(U_2^{(n)}(\alpha)\right)},$$
and 
$$p_1(\alpha, \beta):= \lim_{n\to \infty}\frac{\mu\left(U_1^{(n)}(\alpha)\right)}{\mu\left(U_1^{(n)}(\alpha)\cup U_2^{(n)}(\beta)\cup U_3^{(n)}(1-\alpha-\beta)\right)}$$
$$p_2(\alpha, \beta):= \lim_{n\to \infty}\frac{\mu\left(U_2^{(n)}(\beta)\right)}{\mu\left(U_1^{(n)}(\alpha)\cup U_2^{(n)}(\beta)\cup U_3^{(n)}(1-\alpha-\beta)\right)}$$
$$p_3(\alpha, \beta):= \lim_{n\to \infty}\frac{\mu\left(U_3^{(n)}(1-\alpha-\beta)\right)}{\mu\left(U_1^{(n)}(\alpha)\cup U_2^{(n)}(\beta)\cup U_3^{(n)}(1-\alpha-\beta)\right)}.$$

Then $\XX_0\nleq \uu_n(\tau_1, \tau_2,\tau_3)$ implies, for all large $n$ and $\alpha=\frac{\tau_1}{\tau_1+\tau_2+\tau_3}$, $\beta=\frac{\tau_2}{\tau_1+\tau_2+\tau_3}$, 
\begin{enumerate}
\item with asymptotic probability $p_1(\alpha, \beta)$ we have $x\in U_1^{(n)}(\tau_1)$, in which case $X_{13}<u_{n3}(\tau_3)$ with asymptotic probability $\theta_{1}(\alpha, \beta)$ ($X_{11}<u_{n1}(\tau_1)$ and $X_{12}<u_{n2}(\tau_2)$ with probability 1);
\item with asymptotic probability $p_2(\alpha, \beta)$ we have $x\in U_2^{(n)}(\tau_1)$, in which case $X_{13}<u_{n3}(\tau_3)$ with asymptotic probability $\theta_{2}(\alpha, \beta)$ ($X_{11}<u_{n1}(\tau_1)$ and $X_{12}<u_{n2}(\tau_2)$ with probability 1);
\item with asymptotic probability $p_3(\alpha, \beta)$ we have $x\in U_3^{(n)}(\tau_3) $, in which case $X_{11}<u_{n1}(\tau_1)$, $X_{12}<u_{n2}(\tau_2)$ and $X_{13}<u_{n3}(\tau_3)$ with probability 1.
\end{enumerate}

So it is sufficient to consider $q=1$ and write
$$\theta(\alpha, \beta)= 1-p_1(\alpha, \beta)\left(1-\theta_{1}(\alpha, \beta)\right)- p_2\left(\alpha, 1-\beta)(1-\theta_{2}(\alpha, \beta)\right).$$

In the acip case where the observables all take the same form,
$$\theta_{1}(\alpha, \beta) = \max\left\{0, 1-\frac1{|Df(\zeta_1)|}\frac{1-\alpha-\beta}{\alpha}\frac{\rho(\zeta_1)}{\rho(f(\zeta_1)}\right\},$$
$$\theta_{2}(\alpha, \beta) =  \max\left\{0, 1-\frac1{|Df(\zeta_2)|}\frac{1-\alpha-\beta}{\beta}\frac{\rho(\zeta_2)}{\rho(f(\zeta_2)}\right\},$$
hence
\begin{align*}\theta(\tau_1, \tau_2, \tau_3)&= 1-\alpha\left(1-  \max\left\{0, 1-\frac1{|Df(\zeta_1)|}\frac{1-\alpha-\beta}{\alpha}\frac{\rho(\zeta_1)}{\rho(f(\zeta_1)}\right\} \right) \\
& \qquad- \beta\left(1- \max\left\{0, 1-\frac1{|Df(\zeta_2)|}\frac{1-\alpha-\beta}{\beta}\frac{\rho(\zeta_2)}{\rho(f(\zeta_2)}\right\}\right).
\end{align*}

So again in the doubling map case with Lebesgue where the observables are all of the same form,
$$\theta(\tau_1, \tau_2, \tau_3)= 1-\alpha\left(1-\max\left\{0, 1-\frac{1-\alpha-\beta}{2\alpha}\right\}\right) -\beta\left(1-\max\left\{0, 1-\frac{1-\alpha-\beta}{2\beta}\right\}\right).$$
In this case $\theta_1=\theta_2=\theta_3=1$, $\Gamma(\ttau) = \theta(\ttau)\hat\Gamma(\ttau)$ and consequently $D(\alpha_1,\alpha_2)$ shares the same expression as $\theta$, with $\alpha_1=\alpha$ and $\alpha_2=\beta$ (it is easy to show that $\hat\Gamma(\ttau)=\tau_1+\tau_2+\tau_2$), see Figure~\ref{fig:examples23}.

\subsubsection{A periodic case}
\label{ssec:overlap_per}

Suppose that $\zeta_3=f(\zeta_1)=f(\zeta_2)$ and that $f^2(\zeta_1)= f(\zeta_1)$.  We use the notation for the neighbourhoods of sets, the $p_1, p_2, p_3$ and $\theta_{1}, \theta_{2}$ from Section~\ref{ssec:overlap_nonper}.  However we now also need
$$\theta_{3}(\alpha):= \lim_{n\to \infty} \frac{\mu\left(\left(\hat V_1^{(n)}(\alpha)\cup \hat V_2^{(n)}(1-\alpha)\right)\sm f^{-1}\left(\hat V_1^{(n)}(\alpha)\cup \hat V_2^{(n)}(1-\alpha)\right)\right)}{\mu\left( \hat V_1^{(n)}(\tau_1) \cup \hat V_2^{(n)}(\tau_2)\right)}.$$

Then $\XX_0\nleq \uu_n(\tau_1, \tau_2)$ implies, for all large $n$ and $\alpha$ as above, 
\begin{enumerate}
\item with asymptotic probability $p_1(\alpha)$ we have $x\in V_1^{(n)}(\tau_1)$, in which case $X_{12}<u_{n2}(\tau_2)$ with asymptotic probability $\theta_{1}(\alpha)$ ($X_{11}<u_{n1}(\tau_1)$ with probability 1);
\item with asymptotic probability $p_2(\alpha)$ we have $x\in V_2^{(n)}(\tau_2)$, in which case $X_{11}<u_{n1}(\tau_1)$ with asymptotic probability $\theta_{2}(\alpha)$ ($X_{12}<u_{n2}(\tau_2)$ with probability 1);
\item with asymptotic probability $p_3(\alpha)$ we have $x\in \hat V_1^{(n)}(\tau_1) \cup \hat V_2^{(n)}(\tau_2)$, in which case $X_{11}<u_{n1}(\tau_1)$  and $X_{12}<u_{n2}(\tau_2)$ with probability $\theta_3(\alpha)$.
\end{enumerate}
Summing these possibilities we note that we may consider $q=1$ and obtain
\begin{align*}
\theta(\alpha, 1-\alpha)&= p_1(\alpha)\theta_{1}(\alpha)+ p_2(\alpha)\theta_{2}(\alpha)+ (1-p_1(\alpha)-p_2(\alpha))\theta_{3}(\alpha).
\end{align*}

For our concrete case, note that since $\zeta_3$ has three preimages (including itself), we need a map of degree three or more, so suppose $f$ is the tripling map with Lebesgue as invariant measure.  The $p_1(\alpha)$ and $p_2(\alpha)$ are obtained analogously to in  Section~\ref{ssec:overlap_nonper}, and note that in contrast to there, we also need to add $p_3(\alpha)\theta_3(\alpha) = \frac{\max\{\alpha, 1-\alpha\}}{(1+\max\{\alpha, 1-\alpha\})}\frac23$ instead of just $p_3(\alpha)\cdot 1$.
We compute
\[  \theta(\alpha, 1-\alpha)= \begin{cases} \frac43\left(\frac{1-\alpha}{2-\alpha}\right) & \text{ if } \alpha \in [0,1/4],\\
\frac{1}{2-\alpha} & \text{ if } \alpha \in (1/4, 1/2],\\
\frac{1}{1+\alpha} & \text{ if } \alpha \in (1/2 ,3/4],\\
\frac43\left(\frac{\alpha}{1+\alpha} \right)& \text{ if } \alpha \in (3/4, 1],
\end{cases}
 \qquad 
D(\alpha)= \begin{cases} 1-\alpha & \text{ if } \alpha \in [0,1/4],\\
\frac34 & \text{ if } \alpha \in (1/4, 3/4],\\
\alpha & \text{ if } \alpha \in (3/4, 1].
\end{cases} 
\]
Here $\theta_1=\theta_2=2/3$ so $\Gamma(\ttau)= \frac32\theta(\ttau)\hat\Gamma(\ttau)=\frac32\theta(\ttau)\left(\tau_1+\tau_2-\frac12\min\{\tau_1, \tau_2\}\right)$, see Figure~\ref{fig:examples23}.

Observe that in this case, as in Section~\ref{ssec:overlap_nonper}, we have the presence of both spatial and temporal connections between the maximal sets $\mathcal Z_1$ and $\mathcal Z_2$. The Pickands dependence function is graphically similar to the one obtained in Section~\ref{ssec:overlap_nonper}, but one can easily notice that it is closer to the constant function equal to $1$, which serves as upper bound and should be interpreted as independence between the extremal behaviour of the components. The reason for this is tied to the use of the tripling map instead of the doubling map used in the former case. This means that the repelling effect is stronger ($3>2$), which implies that the link created by the clustering fades away faster, making the interdependence weaker and therefore the graph is closer to the upper bound representing independence rather than the lower bound corresponding to perfect association. 

\begin{figure}
\includegraphics[scale=0.35]{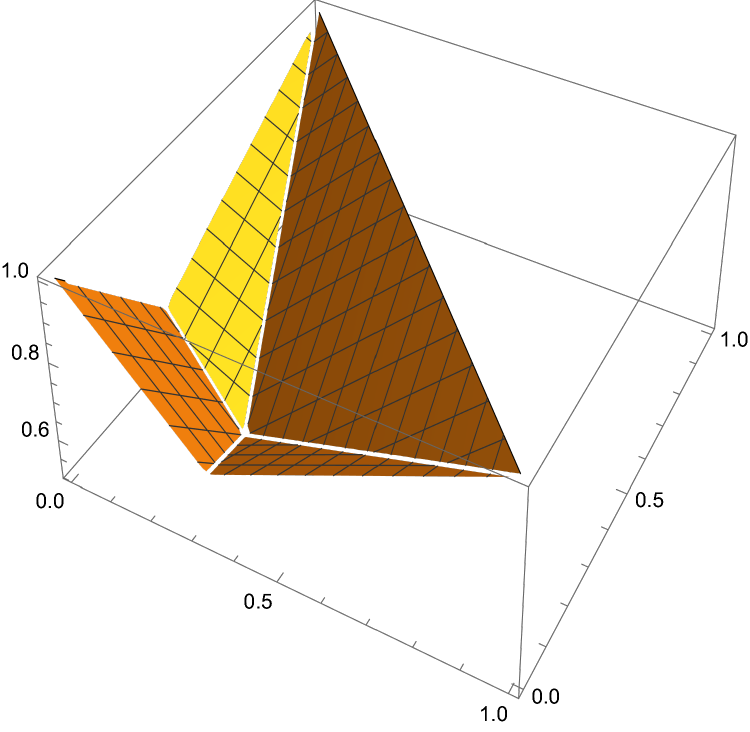}\hspace{0.5cm}\includegraphics[scale=0.35]{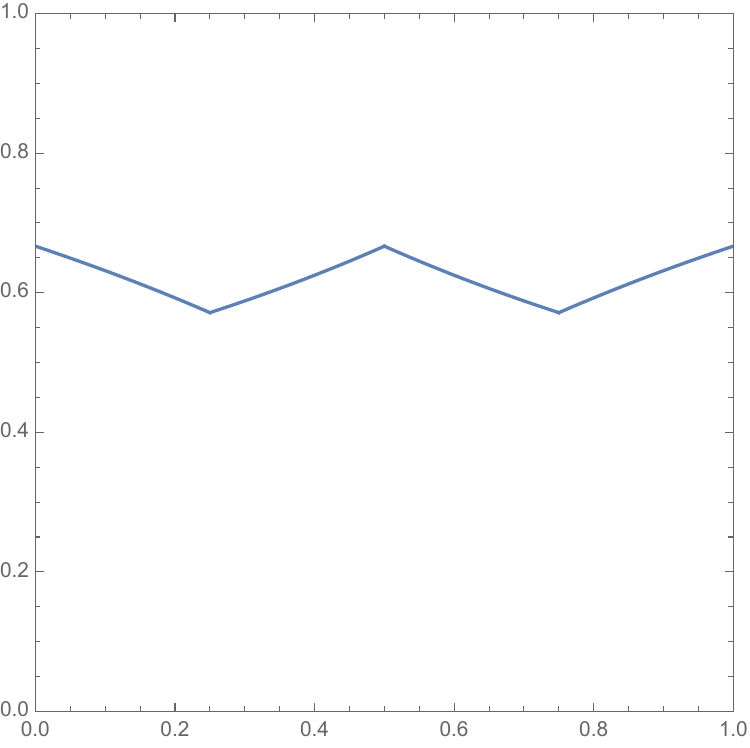}\hspace{0.5cm}\includegraphics[scale=0.35]{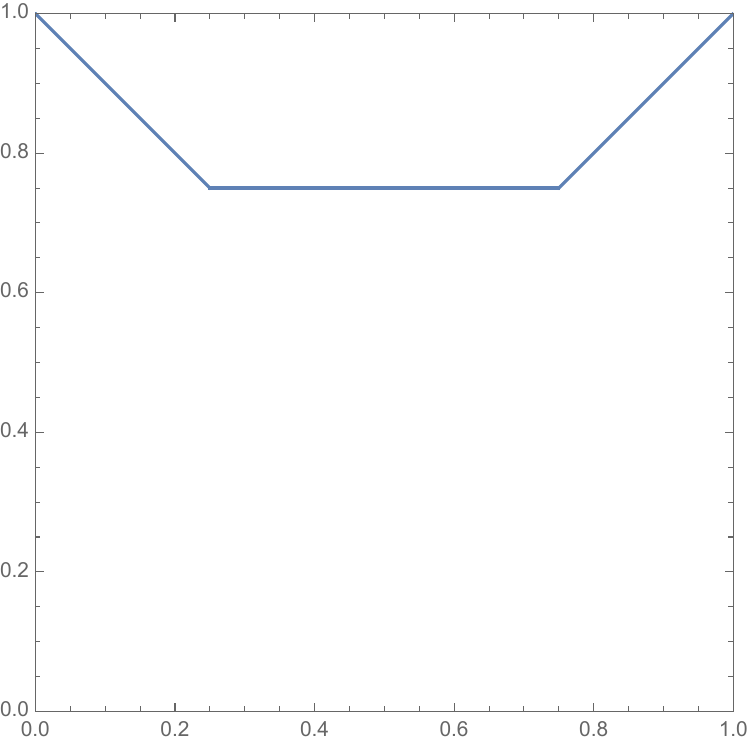}
\caption{Graph of the Pickands dependence function of the example in Section~\ref{sssec:triv} on the left; Graphs of the extremal index function and of the Pickands dependence function of the example in Section~\ref{ssec:overlap_per}, in the middle and on the right, respectively.}
\label{fig:examples23}
\end{figure}

\subsection{A 2d example}

All our previous examples involved one-dimensional uniformly expanding maps, with extremal sets consisting of finitely many points. In what follows, we present a simple example showing that it is also possible to consider higher-dimensional hyperbolic maps with more general extremal sets. More importantly, this example highlights that the Pickands dependence function need not be symmetric, unlike in the previous cases.

Here we consider the toral automorphism $f: \mathbb{T}^2\to \mathbb{T}^2$ induced by $\left(\begin{smallmatrix} 2&1\\ 1&1 \end{smallmatrix}\right)$, i.e. Arnold's cat map.  This has a fixed point at 0 and eigenvalues $\lambda, 1/\lambda$ where $\lambda= \frac{3+\sqrt 5}2$.  Let $\mu$ be Lebesgue measure.  We take two small pieces of local unstable manifold $\cZ_1=(-\eps,\eps)\times\{0\}$, $\cZ_2=(\lambda\eps,\lambda^2\eps)\times\{0\}$, written in the local coordinate axes $E^u,E^s$, with origin at $x_0$, which correspond to its respective local unstable and stable manifolds.  Let $\psi_1(x) = g_1(d(x, \cZ_1))$ and $\psi_2(x) = g_2(d(x, \cZ_2))$ for $g_i$ as in \eqref{eq:types}. 

Even though this particular example is not covered by Proposition \ref{prop:applications}, we observe that in this invertible case, conditions $\D(\uu_n)$ and $\D'(\uu_n)$ can be checked with a simple adaptation of the argument used in \cite{CFFH15} for the respective univariate versions. Hence, we have:

\begin{proposition}
    \label{prop:cat_map}
    For the toral automorphism $f : \mathbb{T}^2 \to \mathbb{T}^2$, the extremal sets $\cZ_i$ and the observables $\psi_i$ introduced above, the assumptions of Theorem \ref{thm:distributional-convergence} hold and in particular we have 
$$
\lim_{n\to\infty}\mu(\MM_n\leq \uu_n(\ttau))=\e^{-\theta(\ttau)\hat\Gamma(\ttau)},
$$
where $\hat\Gamma(\ttau)$ and $\theta(\ttau)$ are given by \eqref{eq:unvector} and \eqref{def:thetan} respectively.
\end{proposition}

For $n$ large enough, $U_1^{(n)}(\tau_1)$ and $U_2^{(n)}(\tau_2)$ are two disjoint rectangles\footnote{\label{footnote1}In fact, for the usual Euclidean metric in the definition of $\PPsi$ these sets have rounded tips. 
However, since these semidisks have an asymptotically negligible measure of the order $\left(\mu\left(U_i^{(n)}\right)\right)^2$, then to simplify both the computations and the diagrams, we will simply disregard these half discs as if the metric would measure distances only in the stable direction. Therefore, we will assume that the sets $U_i^{(n)}$ are actual rectangles, which asymptotically makes no difference.} of widths $2\eps$ and $\lambda(\lambda-1)\eps$, with heights $h_1\sim\frac{\tau_1}{2\eps n}$ and $h_2\sim\frac{\tau_2}{(\lambda^2-\lambda)\eps n}$, respectively.

Note that $V_n(\ttau) = \cap_{i=1}^2 U_i^{(n)}(\tau_i)=\emptyset$ 
 and therefore:
\begin{align*}
\hat\Gamma(\ttau)&=\lim_{n\to\infty}n\mu(\XX_0\not\leq \uu_n(\ttau))=\lim_{n\to\infty}n\mu\left(\bigcup_{i=1}^2\{X_{0i}>u_{ni}(\tau_i)\}\right)\\
&=\lim_{n\to\infty}n\sum_{i=1}^2\mu\left(X_{0i}>u_{ni}(\tau_i)\right)
-\lim_{n\to\infty}n\mu\left(\bigcap_{i=1}^2\{X_{0i}>u_{ni}(\tau_i)\}\right)=\tau_1+\tau_2.
\end{align*}

In order to compute the set $\A$ (with the right choice of $q$), we start by noting that the dynamics pushes the set $U_2^{(n)}(\tau_2)$ along the unstable manifold away from both $U_1^{(n)}(\tau_1)$ and $U_2^{(n)}(\tau_2)$, which means that if the orbit enters $U_2^{(n)}(\tau_2)$, then no short returns to $\{\XX_0\leq \uu_n(\ttau)\}$ are expected. The set  $U_1^{(n)}(\tau_1)$ is shrunk in the stable (vertical) direction and stretched in the unstable (horizontal) and overlaps with $U_1^{(n)}(\tau_1)$ immediately in the first iteration. This means that the middle vertical strip of width $2\lambda^{-1}$ (depicted in blue in Figures \ref{fig:2dcase1} and \ref{fig:2dcase2}) must be removed from $U_1^{(n)}(\tau_1)$. Observe that the right vertical strip (depicted in red in Figures \ref{fig:2dcase1} and \ref{fig:2dcase2}) falls exactly into the empty space between  $U_1^{(n)}(\tau_1)$ and  $U_2^{(n)}(\tau_2)$, however after two iterates it stretches completely horizontally across $U_2^{(n)}(\tau_2)$. Now, one of two possible scenarios can occur, either $h_1<\lambda^2h_2$ and then the red strip is completely contained in $U_2^{(n)}(\tau_2)$ and must be removed from $\A$ (Figure \ref{fig:2dcase1}); or not and then only the central part of the red strip that meets $U_2^{(n)}(\tau_2)$ must be removed (Figure \ref{fig:2dcase2}). Note that the dynamics will push the  possibly surviving points further and further way from the sets $U_1^{(n)}(\tau_1)$ and  $U_2^{(n)}(\tau_2)$. This means that we should take $q=2$ in this case.
\begin{figure}
\includegraphics[width=\textwidth]{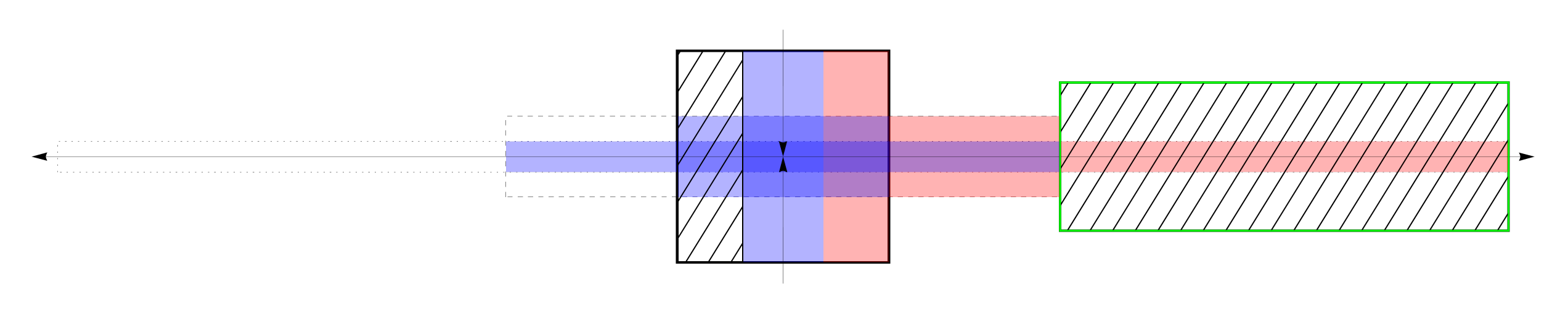}
\caption{$U_1^{(n)}(\tau_1)$ is the box outlined with a black line, $U_2^{(n)}(\tau_2)$ is the box outlined in green, the first iterate of  $U_1^{(n)}(\tau_1)$ corresponds to the black, dashed box, while the second iterate corresponds to the black, dotted box. $\A$ corresponds to the regions covered by slanted lines.
In this case, $h_1<\lambda^2h_2$. Note that for large $n$ the sets $U_i^{(n)}$ are very thin rectangles with rounded tips,  
but for pictorial simplicity we disregard the semidisks as mentioned in footnote \ref{footnote1}.}
\label{fig:2dcase1}
\end{figure}

\begin{figure}
\includegraphics[width=\textwidth]{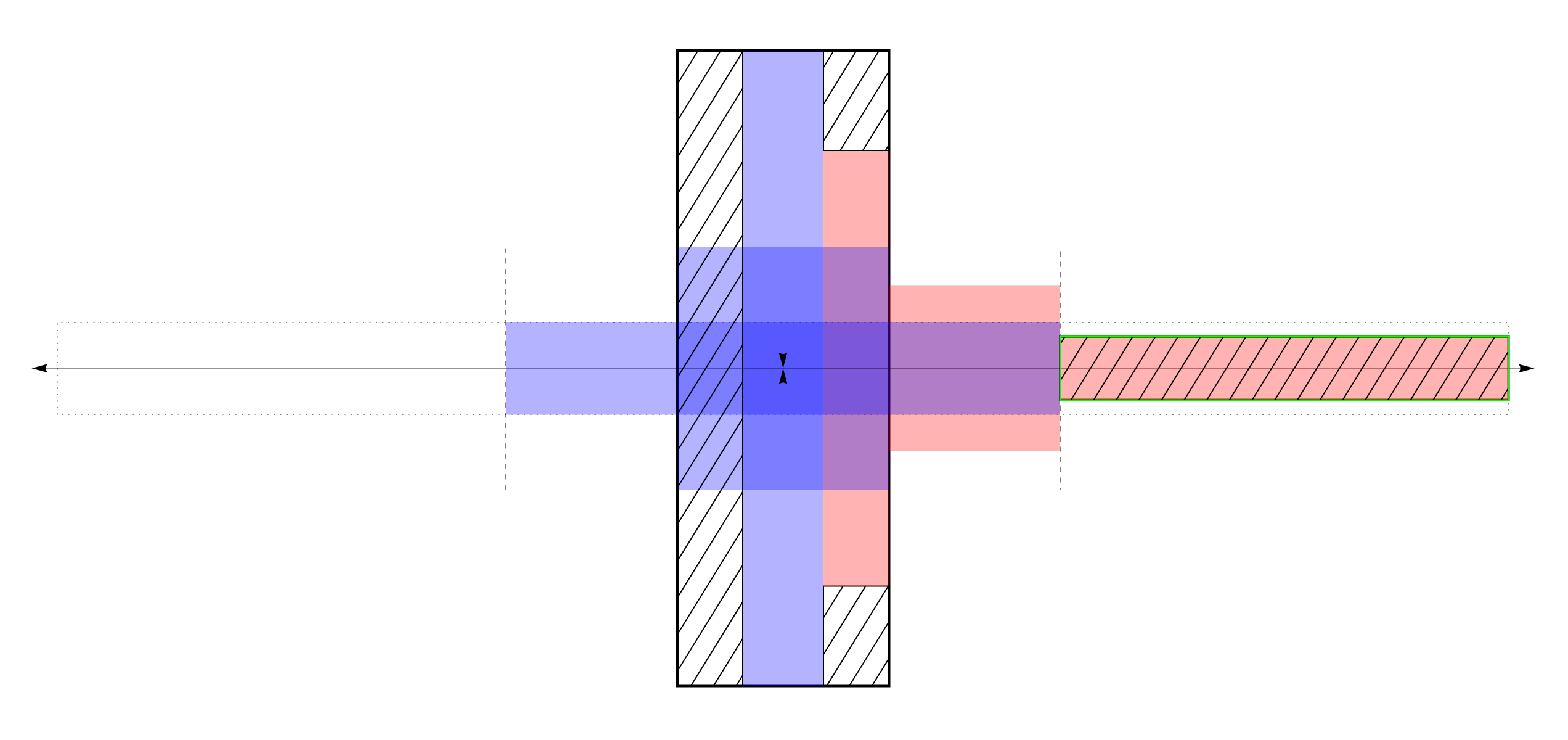}
\caption{$U_1^{(n)}(\tau_1)$ is the box outlined with a black line, $U_2^{(n)}(\tau_2)$ is the green box, the first iterate of  $U_1^{(n)}(\tau_1)$ corresponds to the black, dashed  box, while the second iterate corresponds to the black, dotted box. $\A$ corresponds to the regions covered by slanted lines. In this case, $h_1>\lambda^2h_2$. The same comment regarding the shape of $U_i^{(n)}$ as in the caption of Figure~\ref{fig:2dcase1} applies.}
\label{fig:2dcase2}
\end{figure}

Before we compute the extremal index function we express the turning point $h_1=\lambda^2h_2$ as $\tau_1=\frac{2\lambda\tau_2}{\lambda-1}$ or $\alpha=\frac{2\lambda}{3\lambda-1}$, for $\alpha=\frac{\tau_1}{\tau_1+\tau_2}$. 

When $\tau_1\leq\frac{2\lambda\tau_2}{\lambda-1}$, $A_n^{(q)}$ is the union of two rectangles: one with side lengths $\epsilon(1 - \lambda^{-1})$ and $h_1$, the other with side lengths $\lambda(\lambda-1) \epsilon$ and $h_2$. We thus compute $\theta(\ttau)$ as follows:
\begin{align*}
\theta(\tau_1,\tau_2)&=\lim_{n\to\infty}\frac{\mu\left(U_1^{(n)}(\tau_1)\setminus \left(f^{-1}\left(U_1^{(n)}(\tau_1)\right)\cup f^{-2}\left(U_2^{(n)}(\tau_2)\right)\right)\right)+\mu\left(U_2^{(n)}(\tau_2)\right)}{\mu\left(U_1^{(n)}(\tau_1)\right)+\mu\left(U_2^{(n)}(\tau_2)\right)}\\
&=\lim_{n\to\infty}\frac{(1-\lambda^{-1})\eps h_1+\eps (\lambda^2-\lambda) h_2}{2\eps h_1+\eps (\lambda^2-\lambda) h_2}=\lim_{n\to\infty}\frac{\frac{(1-\lambda^{-1})\tau_1}{2n}+\frac{\tau_2}{n}}{\frac{\tau_1}n+\frac{\tau_2}n}=1-\frac{(1+\lambda^{-1})\alpha}2.
\end{align*}
When  $\tau_1>\frac{2\lambda\tau_2}{\lambda-1}$, we must add two additional rectangles to the previous union, both with side lengths $\epsilon(1- \lambda^{-1})$ and $\frac{h_1 - \lambda^2 h_2}{2}$ , and we thus have, using the same formula:
\begin{align*}
\theta(\tau_1,\tau_2)
&=\lim_{n\to\infty}\frac{(1-\lambda^{-1}) \eps  h_1+(1-\lambda^{-1}) \eps  (h_1-\lambda^2 h_2)+\eps (\lambda^2-\lambda) h_2}{2\eps h_1+\eps (\lambda^2-\lambda) h_2}\\
&=\lim_{n\to\infty}\frac{\frac{(1-\lambda^{-1})\tau_1}{2n}+(1-\lambda^{-1})\left(\frac{\tau_1}{2n}-\frac{\lambda\tau_2}{(\lambda-1)n}\right)+\frac{\tau_2}{n}}{\frac{\tau_1}n+\frac{\tau_2}n}=\frac{(\lambda-1)\alpha}\lambda.
\end{align*}
Observe that $\ttheta=(\theta_1,\theta_2)=(1-\lambda^{-1},1)$, so
\[  \theta(\alpha, 1-\alpha)= \begin{cases} 
1-\frac{(1+\lambda^{-1})\alpha}2 & \text{ if } \alpha\in \left[0,\frac{2\lambda}{3\lambda-1}\right],\\
\frac{(\lambda-1)\alpha}\lambda& \text{ if } \alpha\in \left(\frac{2\lambda}{3\lambda-1},1\right],
\end{cases}
 \qquad D(\alpha)= \begin{cases} 1-\frac\alpha2 & \text{ if } \alpha \in [0,2/3],\\
\alpha & \text{ if } \alpha \in (2/3, 1].
\end{cases}
\]
\begin{figure}
\includegraphics[scale=0.4]{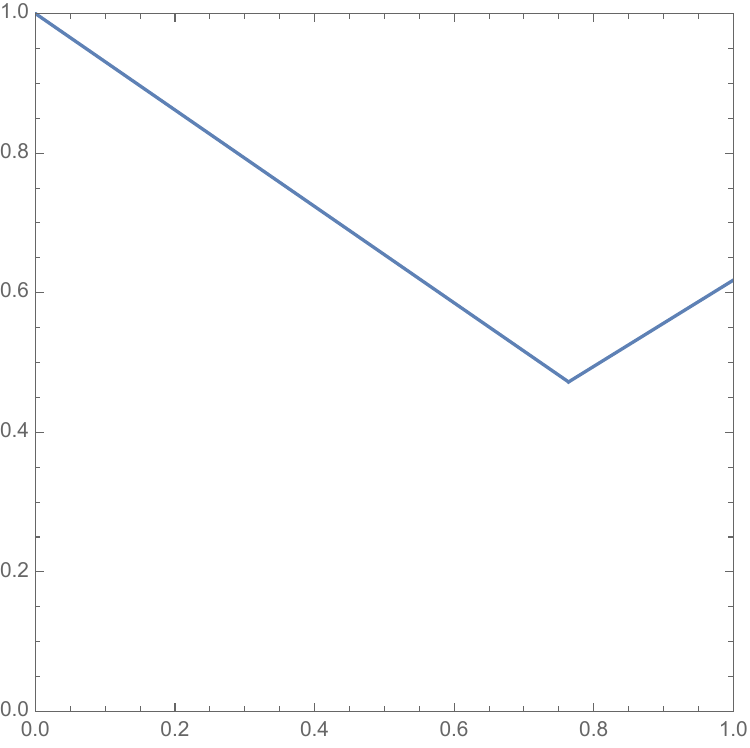}\hspace{2cm}\includegraphics[scale=0.4]{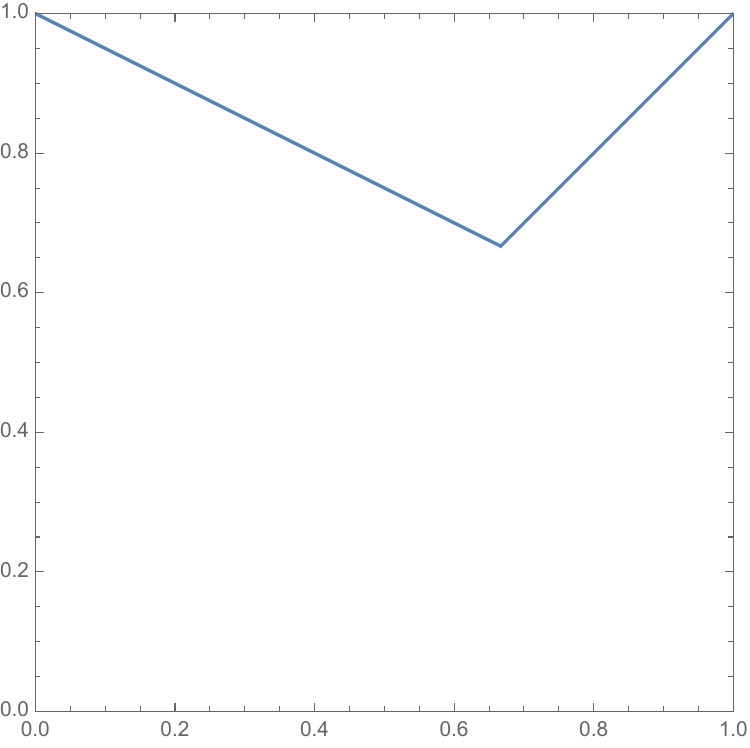}
\caption{On the left is the graph of the extremal index function and on the right is the graph of the Pickands $D$ function, both associated to $\Gamma$ for the cat map example.}
\label{fig:2dgraphs}
\end{figure}
Observe that in this case, we have no spatial connection as the maximal sets have an empty intersection, which immediately translates into the usual formula associated to independence for $\hat\Gamma$. However, there exists a temporal link between the maximal sets, which creates clustering and dependence between the components extremal limiting behaviour. The asymmetry observed in Pickands dependence function is due to the fact that an exceedance in the first component is likely to be followed by one in the second component, but not the other way around. 

We also remark in the illustrative examples we considered here, the nice scaling properties of Lebesgue measure, the piecewise linear nature of the maps considered and the regularity of the observables came together to give us piecewise linear Pickands dependence functions, but the mechanisms used can provide more general shapes, which will be further investigated in forthcoming works.

\appendix

\section{Dependence functions and copulas}
\label{sec:appendix_copulas}

In the first classical works of multivariate Extreme Value Theory, the dependence between the components was described by copulas, whose relation with the stable dependence function we clarify here. We start by introducing the copula of a multivariate \df $\FF$, following \cite{H89a} closely. 
\begin{definition}
Let $\mathbb Y=(Y_1,\ldots Y_d)$ be a random vector defined on a probability space $(\Omega,\mathcal B,\p)$ with \df $\FF$. We define the \emph{copula} $C_{\FF}$ as a multivariate \df supported on $[0,1]^d$ and such that
$$
C_{\FF}(\tt)=\p(F_1(Y_1)\leq t_1,\ldots, F_d(Y_d)\leq t_d),
$$
where $(F_j)_j$ denote the marginals of $\FF$. 
\end{definition}
The copula describes how the dependence between the components affects the joint distribution in the sense that we can recover $\FF$ from its marginals: $\FF(\tt)=C_{\FF}(F_1(t_1),\ldots,F_d(t_d))$. 
\begin{definition}
A copula $C$ is said to be an \emph{extreme value copula} if it is the copula associated to a \df arising as a weak limit for the distribution of $\MM_n$.
\end{definition}
\begin{remark}
Observe that both $C_H$ and $C_{\hat H}$ are extreme value copulas and since $\hat H$ has uniform marginals then $C_{\hat H}=\hat H$ (see \cite[Section~1]{N94}).
\end{remark}
In the \iid setting, letting $\FF$ denote the \df of $\hat \XX_0$, then $\FF_n(\tt):=(\FF(\tt))^n$ is the \df of $\hat \MM_n$ and one can also show that for every $n\in \N$, we have $C_{\FF_n}(\tt)=\left(C_{\FF}(\tt^{1/n})\right)^n$ (see \cite[Lemma~2.2]{H89a}), which eventually leads to the homogeneity property:
\begin{equation}
\label{eq:copula-homogeneity}
C_{\hat H}(\tt^c)=\left(C_{\hat H}(\tt)\right)^c \qquad \text{for all}\quad \tt\in[0,1]^d\quad\text{and}\quad c>0.
\end{equation} 
This homogeneity property is often referred to as max-stability and it is easy to see that in the \iid setting the class of extreme value copulas coincides with that of max-stable copulas. Moreover, weak convergence of a sequence of multivariate distributions can be decomposed into two parts, one corresponding to the convergence of the marginals (which is a univariate problem) and the convergence of the copulas (see \cite[Section~3]{H89a} and references therein). Since the weak convergence of the marginals is a univariate problem, the main interest of multivariate analysis lies in understanding the copulas. 

From the above the following lemma follows easily, giving an alternative definition of $\Gamma$.
\begin{lemma}
The stable dependence function associated to an extreme value copula $C$ for $\ttau\in [0,\infty)^d$ can be written as
\begin{equation*}
\label{eq:gamma-hat-def}
\Gamma(\ttau)=-\log C \left(\e^{-\ttau}\right).
\end{equation*}
\end{lemma}

\section{Pickands dependence function}

In this section, we make a brief summary of the properties of the Pickands dependence function on the classical setting of \iid sequences, providing some heuristics on its interpretation. We refer to \cite[Sections 8 and 9]{BGTS04} for further details on the subject and examples.

The Pickands dependence function provides a rather convenient way of describing the extremal component dependence. This is particularly true in the bivariate case, where it provides a nice pictorial insight about the extremal interconnection between the components. 
Since we are in the classical setting, we can write the Pickands function as:
$$
D(\alpha)=\hat\Gamma(\alpha, 1-\alpha), \qquad\text{for $\alpha\in[0,1]$}.
$$  
The convexity of $\hat \Gamma$ and the relationship \eqref{eq:bounds} imply that $D$ is convex and bounded by:
\begin{equation}
\label{eq:bounds-Pickands}
\max\{\alpha,1-\alpha\}\leq D(\alpha)\leq 1.
\end{equation}
The closer the graph of $D$ is to the constant function $1$, the closer we are to component-wise extremal independence, while the closer the graph is to the function  $g(\alpha)=\max\{\alpha,1-\alpha\}$, the closer we are to asymptotic perfect association. 

In the classical setting, many statistical models were introduced to describe extremal component limiting dependence.  In order to illustrate a possible extremal component dependence behaviour, we consider the so called logistic model (see for example \cite[Section~ 9.2.2]{BGTS04}), for which we have 
$\hat\Gamma(\ttau)=\left(\tau_1^{1/\beta}+\tau_2^{1/\beta}\right)^\beta$, where $\beta\in(0,1]$ is a tuning parameter determining the strength of the asymptotic link between the two components. In this case, the Pickands dependence function is given by $D(\alpha)=\left(\alpha^{1/\beta}+(1-\alpha)^{1/\beta}\right)^\beta$. 
\begin{figure}
\includegraphics[width=7cm]{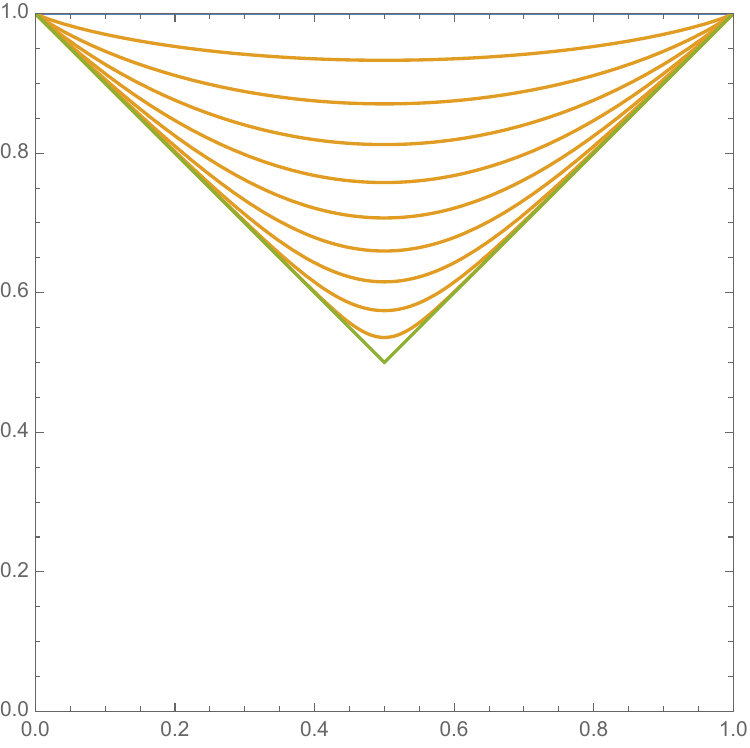}
\caption{Graphs of the upper and lower bounds for the Pickands dependence function given in \eqref{eq:bounds-Pickands} and the Pickands dependence function of the logistic model for parameters $\beta=0.1,0.2,\ldots,0.9$. Note that $\beta=0.1$ corresponds to the lower graph and $\beta=0.9$ the upper one.}
\label{fig:logistic-Pickands}
\end{figure}
Observe that as it can be seen in Figure~\ref{fig:logistic-Pickands} for $\beta$ close to 1 we get asymptotic independence and for $\beta$ close to 0 we get perfect association.

\bibliographystyle{amsalpha}

\bibliography{Multivariate.bib}

\end{document}